\documentclass[oneside, 11pt]{amsart} 
\usepackage{amsmath,amsthm,amssymb,epic}
\usepackage{eqlist,eqparbox}
\usepackage{amsfonts}
\usepackage{latexsym}
\usepackage[2emode]{psfrag}
\usepackage{amsthm}
\usepackage{amsmath}
\usepackage[all]{xy}

\begin{document}

\newcommand{\N}{\mbox {$\mathbb N $}}
\newcommand{\Z}{\mbox {$\mathbb Z $}}
\newcommand{\Q}{\mbox {$\mathbb Q $}}
\newcommand{\R}{\mbox {$\mathbb R $}}
\newcommand{\lo }{\longrightarrow }
\newcommand{\ul}{\underleftarrow }
\newcommand{\rl}{\underrightarrow }
\newcommand{\rs }{\rightsquigarrow }
\newcommand{\ra }{\rightarrow }
\newcommand{\dd }{\rightsquigarrow }
\newcommand{\ol }{\overline }
\newcommand{\la }{\langle }
\newcommand{\tr }{\triangle }
\newcommand{\xr }{\xrightarrow }
\newcommand{\de }{\delta }
\newcommand{\pa }{\partial }
\newcommand{\LR }{\Longleftrightarrow }
\newcommand{\Ri }{\Rightarrow }
\newcommand{\va }{\varphi }
\newcommand{\Den}{{\rm Den}\,}
\newcommand{\Ker}{{\rm Ker}\,}
\newcommand{\Reg}{{\rm Reg}\,}
\newcommand{\Fix}{{\rm Fix}\,}
\newcommand{\Img}{{\rm Im}\,}
\newcommand{\Id}{{\rm Id}\,}

\newtheorem{theorem}{Theorem}[section]
\newtheorem{lemma}[theorem]{Lemma}
\newtheorem{proposition}[theorem]{Proposition}
\newtheorem{corollary}[theorem]{Corollary}
\newtheorem{definition}[theorem]{Definition}
\newtheorem{example}[theorem]{Example}
\newtheorem{xca}[theorem]{Exercise}
\theoremstyle{remark}
\newtheorem{remark}[theorem]{Remark}
\numberwithin{equation}{section}

\def\leftmark{L.C. Ciungu}
\title{Very true pseudo-BCK algebras}
\author{Lavinia Corina Ciungu}

\begin{abstract}
In this paper we introduce the very true operators on pseudo-BCK algebras and we study their properties. 
We prove that the composition of two very true operators is a very true operator if and only if they commute. 
Moreover, given a very true bounded pseudo-BCK algebra $(A,v)$, we define the pseudo-BCK$_{vt,st}$ algebra 
by adding two truth-depressing hedges operators associated with $v$. 
We also define the very true deductive systems and the very true homomorphisms and we investigate their properties. 
Also, given a normal $v$-deductive system of a very true pseudo-BCK algebra $(A,v)$ we construct a very true operator 
on the quotient pseudo-BCK algebra $A/H$.  
We investigate the very true operators on some classes of pseudo-BCK algebras such as pseudo-BCK(pP) algebras, FL$_w$-algebras, pseudo-MTL algebras. 
Finally, we define the $Q$-Smarandache pseudo-BCK algebras and we introduce the notion of a very true operator on 
$Q$-Smarandache pseudo-BCK algebras. \\


\end{abstract}

\maketitle

\section{Introduction}

Zadeh was the first to consider the importance of fuzzy truth values as "very true", "quite true", etc.,  
that are in fact fuzzy subsets of all truth degrees (usually, the real interval $[0, 1]$). 
He was interested in methods of handling these fuzzy truth values, but without considering any sort of axiomatization. 
However, an interesting logic problem will be to study if any axiomatization is possible and which methods of mathematical logic could be developed to treat this kind of fuzzy logic. 
$\rm H\acute{a}jek$ introduced in \cite{Haj1} a complete axiomatization of a logic which extends BL-logic by a unary connective "\emph{$vt$}" which can be interpreted as "very true". 
Such a connective is a subdiagonal and monotone operator defined on the set of truth degrees. 
Subdiagonality "$v(x)\le x$" for each truth degree $x$ means that each interpretation $v$ of $vt$ is \emph{truth-stressing}. 
The concept of very true operator defined by $\rm H\acute{a}jek$ is in fact the same as the concept of \emph{hedge} 
introduced by Zadeh in \cite{Zad2} and \cite{Zad1} (see also \cite{Cha2}) and it is a tool for reducing the 
number of possible logical values in multiple-valued fuzzy logic. 
This tool was used by $\rm B\check{e}lohl\acute{a}vek$ and Vychodil for reducing the size of fuzzy concept lattices 
(\cite{Bel1}).  
In this context we mention that an axiomatization of connectives "slightly" and "more or less true" extending the 
propositional BL-logic have been defined by Vychodil in \cite{Vyc1} as a superdiagonal and monotone  
truth operator called a "\emph{truth-depressing hedge}". 
A weaker axiomatization over any core fuzzy logic for both the truth-stressing and truth-depressing connectives, including standard completeness results was presented in \cite{Est1}. 
Recently, in \cite{Le1} two axiomatizations were introduced over any propositional core fuzzy logic for multiple truth-stressing and truth-depressing hedges. 
The notion of a very true operator was extended to other fuzzy logic algebras such as MV-algebras (\cite{Leu1}), effect algebras (\cite{Cha1}), R$\ell$-monoids (\cite{Rac7}), MTL-algebras (\cite{Wang1}), residuated lattices (\cite{Liu1},\cite{Liu2}) and equality algebras (\cite{Wang2}). \\

In this paper we introduce the very true operators on pseudo-BCK algebras and we study their properties. 
We prove that the composition of two very true operators is a very true operator if and only if they commute. 
It is proved that, if two very true operators have the same image, then the two operators coincide. 
For a very true bounded pseudo-BCK algebra $(A,v)$, we define the pseudo-BCK$_{vt,st}$ algebra by adding two truth-depressing hedges operators associated with $v$. 
Given a very true operator on a good pseudo-BCK algebra $A$ satisfying Glivenko property we define  
very true operators on $\Reg(A)$ and $A/\Den(A)$. It is also proved that the composition of a pseudo-valuation 
with a very true operator is a pseudo-valuation as well. 
We define the very true deductive systems of a very true pseudo-BCK algebra $(A,v)$ and we prove that the congruences 
of $(A,v)$ coincide with the congruences of $A$. We also introduce the notion of a very true homomorphism and we 
investigate their properties. Given a normal $v$-deductive system of a very true pseudo-BCK algebra $(A,v)$ we construct  
a very true operator on the quotient pseudo-BCK algebra $A/H$.  
We investigate the very true operators on some classes of pseudo-BCK algebras such as pseudo-BCK(pP) algebras, FL$_w$-algebras, pseudo-MTL algebras. 
We define the $Q$-Smarandache pseudo-BCK algebras and we introduce the notion of a very true operator on 
$Q$-Smarandache pseudo-BCK algebras. Given a $Q$-Smarandache pseudo-BCK algebra $A$ and a very true operator $v$  
on $A$, we prove that the restriction of $v$ to $Q$ is a very true operator on $Q$. 
Additionally we define and investigate the interior operators on pseudo-BCK algebras.  
For any interior operator on a good pseudo-BCK algebra $A$ satisfying Glivenko property we construct   
interior operators on $\Reg(A)$ and $A/\Den(A)$.

\bigskip

\section{Preliminaries}

Pseudo-BCK algebras were introduced by G. Georgescu and A. Iorgulescu in \cite{Geo15} as algebras 
with "two differences", a left- and right-difference, and with a constant element $0$ as the least element. Nowadays pseudo-BCK algebras are used in a dual form, with two implications, $\ra$ and $\rs$ and with one constant element $1$, that is the greatest element. Thus such pseudo-BCK algebras are in the "negative cone" and are also called "left-ones". Pseudo-BCK algebras were intensively studied in \cite{Ior14}, \cite{Ior1}, \cite{Kuhr6}, \cite{Kuhr1}, \cite{Ciu4}.  
In this section we recall some basic notions and results regarding pseudo-BCK algebras. 

\begin{definition} \label{psBE-40-10} $\rm($\cite{Geo15}$\rm)$ A \emph{pseudo-BCK algebra} (more precisely, \emph{reversed left-pseudo-BCK algebra}) is a structure ${\mathcal A}=(A,\leq,\rightarrow,\rightsquigarrow,1)$ where $\leq$ is a binary relation on $A$, $\rightarrow$ and $\rightsquigarrow$ are binary operations on $A$ and $1$ is an element of $A$ satisfying, for 
all $x, y, z \in A$, the  axioms:\\
$(psBCK_1)$ $x \rightarrow y \leq (y \rightarrow z) \rightsquigarrow (x \rightarrow z)$, $\:\:\:$
            $x \rightsquigarrow y \leq (y \rightsquigarrow z) \rightarrow (x \rightsquigarrow z);$ \\
$(psBCK_2)$ $x \leq (x \rightarrow y) \rightsquigarrow y$,$\:\:\:$ $x \leq (x \rightsquigarrow y) \rightarrow y;$ \\
$(psBCK_3)$ $x \leq x;$ \\
$(psBCK_4)$ $x \leq 1;$ \\
$(psBCK_5)$ if $x \leq y$ and $y \leq x$, then $x = y;$ \\
$(psBCK_6)$ $x \leq y$ iff $x \rightarrow y = 1$ iff $x \rightsquigarrow y = 1$.
\end{definition}

A pseudo-BCK algebra is said to be \emph{proper} if it is not a BCK-algebra. \\
In a pseudo-BCK algebra $(A,\ra,\rs,1)$, one can define a binary relation $``\leq"$ by \\
$\hspace*{2cm}$ $x\leq y$ iff $x\ra y=1$ iff $x\rs y=1$, for all $x, y\in A$. \\
If $(A,\ra,\rs,1)$ is a pseudo-BCK algebra satisfying $x\ra y=x\rs y$, for all $x, y\in A$, then it is a BCK-algebra. \\
If there is an element $0$ of a pseudo-BCK algebra $(A, \ra, \rs, 1)$, such that $0\le x$ 
(i.e. $0\rightarrow x=0\rightsquigarrow x=1$), for all $x\in A$, then the pseudo-BCK algebra is said to be  
\emph{bounded} and it is denoted by $(A, \ra, \rs, 0, 1)$. 
In a bounded pseudo-BCK algebra $(A, \ra, \rs, 0, 1)$ we define two negations: \\
$\hspace*{3cm}$ $x^{-}:= x\ra 0$, $x^{\sim}:= x\rs 0$, \\
for all $x\in A$. 
Obviously $x^{-\sim}= x\vee_1 0$ and $x^{\sim-}= x\vee_2 0$. \\ 
If $(A, \ra, \rs, 0, 1)$ is a bounded pseudo-BCK algebra we denote: \\ 
$\hspace*{2cm}$ $\Reg(A)=\{x\in A \mid x^{-\sim}=x^{\sim-}= x\}$, the set of all \emph{regular} elements of $A$, \\
$\hspace*{2cm}$ $\Den(A)=\{x\in A \mid x^{-\sim}=x^{\sim-}= 1\}$, the set of all \emph{dense} elements of $A$. \\ 
If $\Reg(A)=A$, then $A$ is said to be \emph{involutive}. 
If a bounded pseudo-BCK algebra $A$ satisfies $x^{-\sim}=x^{\sim-}$ for all $x\in A$, then $A$ is called a \emph{good} pseudo-BCK algebra. \\
Obviously, if $A$ is involutive, then $A$ is good and $\Den(A)=\{1\}$.   

We will refer to $(A,\ra,\rs,1)$ by its universe $A$. 

\begin{lemma} \label{psBE-40-20} $\rm($\cite{Geo15}$\rm)$ Let $(A,\rightarrow,\rightsquigarrow,1)$ be a pseudo-BCK algebra. Then the following hold for all $x, y, z\in A$: \\
$(1)$ $x\rightarrow (y\rightsquigarrow z)=y\rightsquigarrow (x\rightarrow z);$ \\
$(2)$ $x \leq y$ implies $y \rightarrow z \leq x \rightarrow z$ and
      $y \rightsquigarrow z \leq x \rightsquigarrow z;$ \\
$(3)$ $x \leq y$ implies $z \rightarrow x \leq z \rightarrow y$ and
      $z\rightsquigarrow x \leq z \rightsquigarrow y;$ \\
$(4)$ $x\rightarrow y\le (z\rightarrow x)\rightarrow (z\rightarrow y)$ and 
      $x\rightsquigarrow y\le (z\rightsquigarrow x)\rightsquigarrow (z\rightsquigarrow y)$. 
\end{lemma}

\begin{proposition} \label{psBE-100}$\rm($\cite{{Geo15}}$\rm)$ Let $(A, \ra, \dd, 0, 1)$ be a bounded pseudo-BCK 
algebra. Then the following hold for all $x, y\in A:$ \\
$(1)$ $x\le x^{-\sim}$, $x\le x^{\sim-};$ \\
$(2)$ $x\ra y^{\sim}=y\rs x^{-}$ and $x\rs y^{-}=y\ra x^{\sim};$ \\
$(3)$ $x^{\sim}\ra y^{-\sim}=y^{-}\rs x^{\sim-}$ and $x^{-}\rs y^{\sim-}=y^{\sim}\ra x^{-\sim};$ \\
$(4)$ $x\le y$ implies $y^{-}\le x^{-}$, $y^{\sim}\le x^{\sim}$, $x^{-\sim}\leq y^{-\sim}$ and 
      $x^{\sim-}\leq y^{\sim-};$ \\  
$(5)$ $x^{-\sim-}=x^{-}$ and $x^{\sim-\sim}=x^{\sim};$ \\
$(6)$ $x\ra y^{-\sim}=y^{-}\rs x^{-}=x^{-\sim}\ra y^{-\sim}$ and
               $x\rs y^{\sim-}=y^{\sim}\ra x^{\sim}=x^{\sim-}\rs y^{\sim-};$ \\
$(7)$ $x\ra y^{\sim}=y^{\sim-}\rs x^{-}=x^{-\sim}\ra y^{\sim}$ and
                $x\rs y^{-}=y^{-\sim}\ra x^{\sim}=x^{\sim-}\rs y^{-};$ \\
$(8)$ $(x\ra y^{\sim-})^{\sim-}=x\ra y^{\sim-}$ and $(x\rs y^{-\sim})^{-\sim}=x\rs y^{-\sim};$ \\
$(9)$ $x\ra y\le y^{-}\rs x^{-}$ and $x\rs y\le y^{\sim}\ra x^{\sim}$.  
\end{proposition}

\begin{remark} \label{comm-psBE-50} Pseudo BCK-logic was defined by J. $\rm K \ddot{u}hr$ 
(\cite[Definition 1.3.1]{Kuhr6}). Formulas of pseudo BCK-logic are built from propositional variables and 
the primitive connectives $\ra$ and $\rs$. The axioms are the following formulas:\\
$(B_1)$ $(\varphi \ra \psi)\ra ((\psi \ra \vartheta)\rs (\varphi \ra \vartheta))$, \\
$(B_2)$ $(\varphi \rs \psi)\ra ((\psi \rs \vartheta)\ra (\varphi \rs \vartheta))$, \\
$(C_1)$ $(\varphi \ra (\psi \rs \vartheta)) \ra (\psi \rs (\varphi \ra \vartheta))$, \\
$(C_2)$ $(\varphi \rs (\psi \ra \vartheta)) \ra (\psi \ra (\varphi \rs \vartheta))$, \\ 
$(K_1)$ $\varphi \ra (\psi \ra \varphi)$, \\
$(K_2)$ $\varphi \rs (\psi \rs \varphi)$. \\
The inference rules are the following:\\
$(MP)$ $\frac{\varphi, \varphi \ra \psi}{\psi}$, i.e. from $\varphi$ and $\varphi \ra \psi$ we infer $\psi$, \\
$(IMP_1)$ $\frac{\varphi \ra \psi}{\varphi \rs \psi}$, i.e. from $\varphi \ra \psi$ we infer $\varphi \rs \psi$, \\
$(IMP_2)$ $\frac{\varphi \rs \psi}{\varphi \ra \psi}$, i.e. from $\varphi \rs \psi$ we infer $\varphi \ra \psi$.  
\end{remark}

\begin{example} \label{psBE-50} $\rm($\cite{Bor2}$\rm)$
Consider the set $A=\{a,b,c,1\}$ and the operations $\rightarrow,\rightsquigarrow$ given by the following tables:
\[
\hspace{10mm}
\begin{array}{c|ccccc}
\rightarrow & 1 & a & b & c \\ \hline
1 & 1 & a & b & c \\ 
a & 1 & 1 & 1 & 1 \\ 
b & 1 & a & 1 & c \\ 
c & 1 & b & 1 & 1
\end{array}
\hspace{10mm} 
\begin{array}{c|ccccc}
\rightsquigarrow & 1 & a & b & c \\ \hline
1 & 1 & a & b & c \\ 
a & 1 & 1 & 1 & 1 \\ 
b & 1 & c & 1 & c \\ 
c & 1 & c & 1 & 1
\end{array}
. 
\]
Then $(A,\rightarrow,\rightsquigarrow,a,1)$ is a bounded pseudo-BCK algebra.  
\end{example}

\begin{example} \label{psBE-50-10-10} $\rm($\cite{Ciu3}$\rm)$
Consider the structure $(A,\ra,\rs,1)$, where the operations $\ra$ and $\rs$ on $A=\{1,a,b,c,d,e\}$ 
are defined as follows:
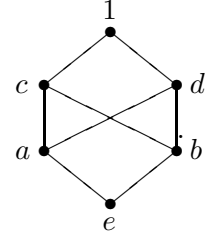
\[
\begin{array}{c|cccccc}
\ra & 1 & a & b & c & d & e \\ \hline
1 & 1 & a & b & c & d & e \\
a & 1 & 1 & d & 1 & 1 & d \\
b & 1 & c & 1 & 1 & 1 & c \\
c & 1 & a & d & 1 & d & a \\
d & 1 & c & b & c & 1 & b \\
e & 1 & 1 & 1 & 1 & 1 & 1 
\end{array}
\hspace{10mm}
\begin{array}{c|cccccc}
\rs & 1 & a & b & c & d & e \\ \hline
1 & 1 & a & b & c & d & e \\
a & 1 & 1 & c & 1 & 1 & c \\
b & 1 & d & 1 & 1 & 1 & d \\
c & 1 & d & b & 1 & d & b \\
d & 1 & a & c & c & 1 & a \\
e & 1 & 1 & 1 & 1 & 1 & 1
\end{array}
\qquad\quad
\begin{picture}(50,-70)(0,30)
\drawline(0,25)(25,5)(50,25)(0,50)(0,25)(50,50)(50,25)
\drawline(0,50)(25,70)(50,50)
\put(0,25){\circle*{4}}
\put(25,5){\circle*{4}}
\put(50,25){\circle*{4}}
\put(0,50){\circle*{4}}
\put(50,50){\circle*{4}}
\put(25,70){\circle*{4}}
\put(22,-5){$e$}
\put(22,75){$1$}
\put(-11,22){$a$}
\put(-11,47){$c$}
\put(55,22){$b$}
\put(55,47){$d$}
\end{picture}
.
\]
Then $(A,\ra,\rs,e,1)$ is an involutive pseudo-BCK algebra. 
\end{example}

\begin{definition} \label{psBE-110} A good pseudo-BCK algebra $A$ has the \emph{Glivenko properties} if it satisfies 
the following conditions for all $x, y\in A:$ \\
$\hspace*{3cm}$ $(x\ra y)^{-\sim}=x\ra y^{-\sim}$ and $(x\rs y)^{-\sim}=x\rs y^{-\sim}$.
\end{definition}

\begin{remark} \label{psBE-110-10} Let $A,\ra,\rs,0,1)$ a good pseudo-BCK algebra satisfying Glivenko property. \\
$(1)$ Applying Proposition \ref{psBE-100}$(6)$ we get: \\
$\hspace*{3cm}$ $(x\ra y)^{-\sim}=x^{-\sim}\ra y^{-\sim}$ and $(x\rs y)^{-\sim}=x^{-\sim}\rs y^{-\sim}$, \\ 
for all $x, y\in A$. \\
$(2)$ $(\Reg(A),\ra,\rs,0,1)$ is a subalgebra of $A$. \\ 
$(3)$ $\Den(A)\in \mathcal{DS}_n(A)$.  
\end{remark} 

Obviously any involutive pseudo-BCK algebra has the Glivenko property. 

A subset $D$ of a pseudo-BCK algebra $A$ is called a \emph{deductive system} of $A$ if it satisfies 
the following axioms: \\
$(ds_1)$ $1\in D,$ \\
$(ds_2)$ $x\in D$ and $x\ra y\in D$ imply $y\in D$. \\
A subset $D$ of $A$ is a deductive system if and only if it satisfies $(ds_1)$ and the axiom: \\
$(ds^{\prime}_2)$ $x\in D$ and $x\rs y\in D$ imply $y\in D$. \\
Denote by ${\mathcal DS}(A)$ the set of all deductive systems of $A$. \\
A deductive system $D$ of $A$ is \emph{proper} if $D\ne A$. \\
A deductive system $D$ of a pseudo-BCK algebra $A$ is said to be \emph{normal} if it satisfies the condition:\\
$(ds_3)$ for all $x, y \in A$, $x \ra y \in D$ iff $x \rs y \in D$. \\
Denote by ${\mathcal DS_n}(A)$ the set of all normal deductive systems of $A$. \\
For details regarding deductive systems and congruence relations on a pseudo-BCK algebra we refer the reader to \cite{Kuhr6}. 
Denote by $\mathcal{CON}(A)$ the set of all congruences on $A$. \\ 
If $\theta\in \mathcal{CON}(A)$, then $H_{\theta}=\{x\in A\mid (x,1)\in \theta \}\in \mathcal{DS}_n(A)$. \\ 
Given $H\in {\mathcal DS}_n(A)$, the relation $\Theta_H$ on $A$ defined by $(x,y)\in \Theta_H$ iff 
$x\rightarrow y\in H$ and $y\rightarrow x\in H$ is a congruence on $A$. 
We write $x/H=[x]_{\Theta_H}$ for every $x\in A$ and we have $H=[1]_{\Theta_H}$.  
Then $(A/\Theta_H,\ra,\rs,[1]_{\Theta_H})$ is a pseudo-BCK algebra called \emph{quotient pseudo-BCK algebra via $H$} and denoted by $A/H$. \\
The function $\pi_H: A \longrightarrow A/H$ defined by $\pi_H(x)=x/H$ for any $x\in A$ is a surjective homomorphism which is called the \emph{canonical projection} from $A$ to $A/H$. \\
One can easily prove that $\Ker(\pi_H)=H$. \\
Let $A, B$ be two pseudo-BCK algebras. A map $f: A\longrightarrow B$ is called a \emph{pseudo-BCK homomorphism} 
if $f(x\ra y)=f(x)\ra f(y)$ and $f(x\rs y)=f(x)\rs f(y)$, for all $x, y\in A$. \\ 
(We use the same notations for the operations in both pseudo-BCK algebras, but the reader must be aware 
that they are different). \\
If $B=A$, then $f$ is called a \emph{pseudo-BCK endomorphism}. \\
Denote $\mathcal{HOM}(A, B)$ the sets of all pseudo-BCK homomorphisms from $A$ to $B$. \\
One can easily check that, if $f$ is a pseudo-BCK homomorphism, then: \\
$(1)$ $f(1)=1;$ $(2)$ $x\le y$ implies $f(x)\le f(y)$. \\
If $f$ is a bounded pseudo-BCK homomorphism such that $f(0)=0$, then the following hold: \\ 
$(3)$ $f(x^{-})=f(x)^{-};$ $(4)$ $f(x^{\sim})=f(x)^{\sim}$. \\

A pseudo-BCK algebra with the \emph{$\rm($pP$\rm)$ condition} (i.e. with the  \emph{pseudo-product} condition) or
a \emph{pseudo-BCK$\rm($pP$\rm)$ algebra} for short, is a pseudo-BCK algebra
$(A,\leq,\rightarrow,\rightsquigarrow,1)$ satisfying the (pP) condition:\\
(pP) For all $x, y \in A$, $x\odot y$ exists where  \\
$\hspace*{2cm}$ $x\odot y=\min\{z \mid x \leq y \rightarrow z\}=\min\{z \mid y \leq x \rightsquigarrow z\}$.

Any involutive pseudo-BCK algebra is a bounded pseudo-BCK(pP) algebra (\cite{Ior1}). 

A \emph{residuated lattice} is an algebra $(A,\wedge,\vee,\odot,\ra,\rs,1)$ satisfying the following axioms: \\
$(rl_1)$ $(A,\wedge,\vee)$ is a lattice; \\
$(rl_2)$ $(A,\odot,1)$ is a monoid; \\
$(rl_3)$ $x\odot y\le z$ iff $x\le y\ra z$ iff $y\le x\rs z$, for all $x,y,z\in A$ (\emph{pseudo-residuation}). \\
A bounded residuated lattice $(A,\wedge,\vee,\odot,\ra,\rs,0,1)$ is called an \emph{FL$_w$-algebra} or 
\emph{bounded integral residuated lattice}. 
Bounded pseudo-BCK(pP) lattices are categorically isomorphic with FL$_w$-algebras. 
A FL$_w$-algebra $(A,\wedge,\vee,\odot,\ra,\rs,0,1)$ satisfying the \emph{pseudo-prelinearity} condition: \\
$\hspace*{2cm}$ $(x \rightarrow y) \vee (y \rightarrow x) = (x \rightsquigarrow y) \vee (y \rightsquigarrow x) =1$, \\
for all $x,y\in $ is called a \emph{pseudo-MTL algebra}. \\
A FL$_w$-algebra $(A,\wedge,\vee,\odot,\ra,\rs,0,1)$ satisfying the \emph{pseudo-divisibility} condition: \\
$\hspace*{2cm}$$(x\ra y)\odot x=x\odot (x\rs y)=x\wedge y$, \\
for all $x,y\in $ is called a \emph{bounded R$\ell$-monoid} or a \emph{divisible residuated lattice}. \\
If a FL$_w$-algebra $(A,\wedge,\vee,\odot,\ra,\rs,0,1)$ satisfies the pseudo-prelinearity and pseudo-divisibility  conditions, then it is called a \emph{pseudo-BL} algebra. 
An involutive pseudo-BL algebra is a \emph{pseudo-MV algebra} (\cite{Ior1}). 
A FL$_w$-algebra $A$ is a pseudo-MV algebra if and only if it satisfies the identities: \\
$\hspace*{3cm}$ $x\vee y=(x\ra y)\rs y=(x\rs y)\ra y$, \\
for all $x, y\in A$ (\cite{Gal3}). 

\begin{lemma} \label{psBCK-pP-10} $\rm($\cite{Ciu4}$\rm)$ In any pseudo-BCK(pP) $(A,\odot,\ra,\rs,1)$ 
the following hold, for all $x, y, z\in A$: \\
$(1)$ $x \odot y \leq x, y;$ \\
$(2)$ $x\rightarrow y \leq (x \odot z) \rightarrow (y \odot z) \leq x \rightarrow (z \rightarrow y)$ and 
$x \rightsquigarrow y \leq (z \odot x) \rightsquigarrow (z \odot y) \leq x \rightsquigarrow (z \rightsquigarrow y)$. \\
If $A$ is a FL$_w$-algebra, then: \\
$(3)$ $(x\ra y)\odot x\le x\wedge y$ and $(x\odot (x\rs y)\le x\wedge y;$ \\
$(4)$ $(x\ra z)\wedge (y\ra z)=(x\vee y)\ra z$ and $(x\rs z)\wedge (y\rs z)=(x\vee y)\rs z;$ \\
$(5)$ $x\vee y\le ((x\ra y)\rs y)\wedge ((y\ra x)\rs x))$ and $x\vee y\le ((x\rs y)\ra y)\wedge ((y\rs x)\ra x))$. \\
\end{lemma}

\bigskip

\section{Interior operators on pseudo-BCK algebras}

Interior operators were defined and studied for R$\ell$-monoids and residuated lattices in \cite{Rac8}, 
\cite{Rac5} and \cite{Rac6}.  
We define and investigate the interior operators on pseudo-BCK algebras.  
The main results proved in this section are similar to the results proved in \cite{Rac1} for the case of closure operators on ordered sets.   
For any interior operator on a good pseudo-BCK algebra $A$ satisfying Glivenko property we define  
interior operators on $\Reg(A)$ and $A/\Den(A)$. 

\begin{definition} \label{int-op-10} Let $A$ be a pseudo-BCK algebra. A mapping $\varphi:A\longrightarrow A$ is 
called an \emph{interior operator} on $A$ if it satisfies the following conditions for all $x, y\in A:$ \\
$(IO_1)$ $\varphi(x)\le x;$                           $\hspace*{9.9cm}$ \emph{(decreasing)} \\
$(IO_2)$ $x\le y$ implies $\varphi(x)\le \varphi(y);$ $\hspace*{7cm}$   \emph{(monotone)}   \\
$(IO_3)$ $\varphi\varphi(x)=\varphi(x)$.              $\hspace*{9.1cm}$ \emph{(idempotent)} 
\end{definition}

Denote $\mathcal{INTO}(A)$ the set of all interior operators on $A$. 

\begin{remark} \label{int-op-10-10}
If condition $(IO_1)$ is replaced by condition $x\le \varphi(x)$ (\emph{increasing}), then $v$ is called a 
\emph{closure operator}. Denote $\mathcal{CLO}(A)$ the set of all closure operators on $A$.  
\end{remark}

\begin{proposition} \label{int-op-10-20} Let $A$ be a pseudo-BCK algebra and let $\varphi, \psi\in \mathcal{INTO}(A)$.
Then $\varphi\le \psi$ if and only if $\varphi\psi=\varphi$. 
\end{proposition}
\begin{proof} Let $\varphi, \psi\in \mathcal{INTO}(A)$. \\
Suppose that $\varphi\le \psi$ and let $x\in A$. We have $\varphi(x)=\varphi\varphi(x)\le \varphi\psi(x)$. 
Moreover $\varphi\psi(x)\le \psi\psi(x)=\psi(x)\le x$, hence $\varphi\varphi\psi(x)\le \varphi(x)$, 
that is $\varphi\psi(x)\le \varphi(x)$. Thus $\varphi\psi=\varphi$. \\
Conversely, assume that $\varphi\psi=\varphi$, hence $\varphi(x)=\varphi\psi(x)\le \psi(x)$, for all $x\in A$, 
that is $\varphi\le \psi$. 
\end{proof}

\begin{theorem} \label{int-op-20} Let $A$ be a pseudo-BCK algebra and let $\varphi, \psi\in \mathcal{INTO}(A)$. 
The following are equivalent: \\
$(a)$ $\varphi\psi=\psi\varphi;$ \\
$(b)$ $\varphi\psi, \psi\varphi\in \mathcal{INTO}(A);$ \\
$(c)$ $\varphi\psi\varphi\psi=\varphi\psi$ and $\psi\varphi\psi\varphi=\psi\varphi$. 
\end{theorem} 
\begin{proof} Let $\varphi, \psi\in \mathcal{INTO}(A)$. \\
$(a)\Rightarrow (b)$ Suppose that $\varphi\psi=\psi\varphi$. For all $x, y\in A$, we have: \\
$\hspace*{2cm}$ $(1)$ $\varphi\psi(x\le \psi(x)\le x$. \\
$\hspace*{2cm}$ $(2)$ $\varphi\psi\varphi\psi(x)=\varphi\varphi\psi\psi(x)=\varphi\varphi\psi(x)=\varphi\psi(x)$. \\
$\hspace*{2cm}$ $(3)$ $x\le y$ implies $\psi(x)\le \psi(y)$, so $\varphi\psi(x)\le \varphi\psi(y)$. \\
It follows that $\varphi\psi$ satisfies $(IO_1)$, $(IO_2)$, $(IO_3)$, hence $\varphi\psi\in \mathcal{INTO}(A)$. \\
Similarly $\psi\varphi\in \mathcal{INTO}(A)$. \\
$(b)\Rightarrow (c)$ Assume that $\varphi\psi, \psi\varphi\in \mathcal{INTO}(A)$. 
Then we have $\varphi\psi\varphi\psi\le \varphi\psi$, and applying Proposition \ref{int-op-10-20} we get $\varphi\psi\varphi\psi=\varphi\psi$. Similarly $\psi\varphi\psi\varphi=\psi\varphi$. \\
$(c)\Rightarrow (a)$ Since $\varphi\psi\varphi\psi=\varphi\psi$ and $\psi\varphi\psi\varphi=\psi\varphi$, 
then we have: \\
$\hspace*{2cm}$ $\varphi\psi(x)=\varphi\psi\varphi\psi(x)\le \psi\varphi\psi(x)\le \psi\varphi(x)$, \\
$\hspace*{2cm}$ $\psi\varphi(x)=\psi\varphi\psi\varphi(x)\le \varphi\psi\varphi(x)\le \varphi\psi(x)$, \\
for all $x\in A$. Hence $\varphi\psi=\psi\varphi$. 
\end{proof}

For any $\varphi\in \mathcal{INTO}(A)$, denote by $\Fix(\varphi)=\{x\in A\mid \varphi(x)=x\}$. 

\begin{theorem} \label{int-op-30} Let $A$ be a pseudo-BCK algebra and let $\varphi, \psi\in \mathcal{INTO}(A)$.
If $\Fix(\varphi)=\Fix(\psi)$, then $\varphi=\psi$. 
\end{theorem}
\begin{proof}
Since $\varphi\varphi(x)=\varphi(x)$, we have $\varphi(x)\in \Fix(\varphi)=\Fix(\psi)$, that is 
$\psi\varphi(x)=\varphi(x)$, for all $x\in A$. Hence $\psi\varphi=\varphi$. \
Similarly $\varphi\psi=\psi$. \\
From $\varphi(x)\le x$ and $\psi(x)\le x$ we get $\varphi(x)=\psi\varphi(x)\le \psi(x)$ and 
$\psi(x)=\varphi\psi(x)\le \varphi(x)$, respectively. Hence $\varphi(x)=\psi(x)$, for all $x\in A$, 
that is $\varphi=\psi$.  
\end{proof}

\begin{proposition} \label{int-op-50} Let $(A,\ra,\rs,0,1)$ be a pseudo-BCK algebra and let 
$\varphi\in \mathcal{INTO}(A)$. Then the following hold for all $x, y\in A:$ \\ 
$(1)$ $x\ra \varphi(y)\le \varphi(x)\ra y$ and $x\rs \varphi(y)\le \varphi(x)\rs y;$ \\
$(2)$ $\varphi(x\ra y)\le \varphi(x)\ra y$ and $\varphi(x\rs y)\le \varphi(x)\rs y$. \\
If $A$ is bounded, then: \\
$(3)$ $\varphi(0)=0;$ \\
$(4)$ $\varphi(x^{-})\le \varphi(x)^{-}$ and $\varphi(x^{\sim})\le \varphi(x)^{\sim};$ \\
$(5)$ $x\le \varphi(x^{-})^{\sim}$ and $x\le \varphi(x^{\sim})^{-};$ \\
$(6)$ $\varphi(x\ra y)\le y^{-}\rs x^{-}$ and $\varphi(x\rs y)\le y^{\sim}\ra x^{\sim}$.  
\end{proposition}
\begin{proof}
$(1)$ From $\varphi(x)\le x$, by Lemma \ref{psBE-40-20}$(2)$ we have 
$x\ra \varphi(y)\le \varphi(x)\ra \varphi(y)$. \\ 
On the other hand, from $\varphi(y)\le y$ by Lemma \ref{psBE-40-20}$(3)$ we get 
$\varphi(x)\ra \varphi(y)\le \varphi(x)\ra y$. \\
Hence $x\ra \varphi(y)\le \varphi(x)\ra y$ and similarly $x\rs \varphi(y)\le \varphi(x)\rs y$. \\
$(2)$ Since $\varphi(x\ra y)\le x\ra y$ and $\varphi(x)\le x$, by Lemma \ref{psBE-40-20}$(2)$ we get 
$\varphi(x\ra y)\le x\ra y\le \varphi(x)\ra y$. Similarly $\varphi(x\rs y)\le \varphi(x)\rs y$. \\
$(3)$ By $(IO_1)$, $\varphi(0)\le 0$, hence $\varphi(0)=0$. \\
$(4)$ It follows from $(2)$ for $y:=0$. \\
$(5)$ From $\varphi(x^{-})\le x^{-}$, $\varphi(x^{\sim})\le x^{\sim}$ we get 
$x\le x^{-\sim}\le \varphi(x^{-})^{\sim}$ and $x\le x^{\sim-}\le \varphi(x^{\sim})^{-}$, respectively. \\
$(6)$ It follows applying $(IO_1)$ and $(psBCK_1)$ for $z:=0$. 
\end{proof}

\begin{proposition} \label{int-op-60} Let $(A, \ra, \rs, 0, 1)$ be a good pseudo-BCK algebra satisfying Glivenko 
property and let $\varphi\in \mathcal{INTO}(A)$. Define $\tilde{\varphi}:\Reg(A)\longrightarrow \Reg(A)$, by $\tilde{\varphi}(x)=\varphi(x)^{-\sim}$. Then $\tilde{\varphi}\in \mathcal{INTO}(\Reg(A))$.  
\end{proposition}
\begin{proof}
Let $x\in \Reg(A)$. Since $\varphi(x)\le x$, then $\tilde{\varphi}(x)=\varphi(x)^{-\sim}\le x^{-\sim}=x$, 
hence $\tilde{\varphi}$ satisfies $(IO_1)$. 
If $x, y\in \Reg(A)$ such that $x\le y$, then $\varphi(x)\le \varphi(y)$ and 
$\varphi(x)^{-\sim}\le \varphi(y)^{-\sim}$. Thus $\tilde{\varphi}(x)\le \tilde{\varphi}(y)$, that is $(IO_2)$. 
Finally, for all $x\in A$, $\tilde{\varphi}\tilde{\varphi}(x)=\tilde{\varphi}(\varphi(x)^{-\sim})= 
(\varphi(x)^{-\sim})^{-\sim}=\varphi(x)^{-\sim}=\tilde{\varphi}(x)$, hence $(IO_3)$ is satisfied. 
We conclude that $\tilde{\varphi}\in \mathcal{INTO}(\Reg(A))$.
\end{proof}

\begin{proposition} \label{int-op-70} Let $(A, \ra, \rs, 0, 1)$ be a good pseudo-BCK algebra satisfying 
Glivenko property and let $\varphi\in \mathcal{INTO}(A)$. Define $\tilde{\varphi}:A/\Den(A)\longrightarrow A/\Den(A)$,  by $\tilde{\varphi}([x]_{\Den(A)})=[\varphi(x)]_{\Den(A)}$. Then $\tilde{\varphi}\in \mathcal{INTO}(A/\Den(A))$.  
\end{proposition}
\begin{proof} 
We show that $\tilde{\varphi}$ is well defined. 
Applying Glivenko property we can see that $[x]_{\Den(A)}=[x^{-\sim}]_{\Den(A)}$ and 
$[x]_{\Den(A)}=[y]_{\Den(A)}$ iff $x^{-\sim}=y^{-\sim}$. Indeed: \\
$\hspace*{1cm}$ $[x]_{\Den(A)}=[y]_{\Den(A)}$ iff $x\ra y, y\ra x\in \Den(A)$ \\
$\hspace*{4.6cm}$ iff $(x\ra y)^{-\sim}=(y\ra x)^{-\sim}=1$ \\
$\hspace*{4.6cm}$ iff $x^{-\sim}\ra y^{-\sim}=y^{-\sim}\ra x^{-\sim}=1$ \\
$\hspace*{4.6cm}$ iff $x^{-\sim}\le y^{-\sim}$ and $y^{-\sim}\le x^{-\sim}$ iff $x^{-\sim}=y^{-\sim}$. \\ 
If $x, y\in A$ such that $[x]_{\Den(A)}=[y]_{\Den(A)}$, we have: \\
$\hspace*{2cm}$ $\tilde{\varphi}([x]_{\Den(A)})=\tilde{\varphi}([x^{-\sim}]_{\Den(A)})=
[\varphi(x^{-\sim})]_{\Den(A)}=[\varphi(y^{-\sim})]_{\Den(A)}$ \\
$\hspace*{4.1cm}$ $=\tilde{\varphi}([y^{-\sim}]_{\Den(A)})=\tilde{\varphi}([y]_{\Den(A)})$, \\
hence $\tilde{\varphi}$ is well defined. \\
Let $x\in A$. From $\varphi(x)\le x$ we get $[\varphi(x)]_{\Den(A)}\le [x]_{\Den(A)}$, so 
$\tilde{\varphi}([x]_{\Den(A)})\le [x]_{\Den(A)}$, thus $(IO_1)$ is verified. 
For any $x, y\in A$ such that $x\le y$ we have $\varphi(x)\le \varphi(y)$, so 
$[\varphi(x)]_{\Den(A)}\le [\varphi(y)]_{\Den(A)}]$. It follows that  
$\tilde{\varphi}([x]_{\Den(A)})\le \tilde{\varphi}([y]_{\Den(A)})$, hence $\tilde{\varphi}$ satisfies $(IO_2)$. 
Let $x\in A$, so $\varphi\varphi(x)=\varphi(x)$ and [$\varphi\varphi(x)]_{\Den(A)}=[\varphi(x)]_{\Den(A)}$. 
Thus $\tilde{\varphi}\tilde{\varphi}([x]_{\Den(A)})=\tilde{\varphi}([x]_{\Den(A)})$, so $(IO_3)$ is satisfied. 
Hence $\tilde{\varphi}\in \mathcal{INTO}(A/\Den(A))$. 
\end{proof}

\begin{example} \label{int-op-80} Let $(A, \ra, \rs, 1)$ be the pseudo-BCK algebra from 
Example \ref{psBE-50} and the maps $\varphi_i:A\longrightarrow A$, $i=1,\cdots,8$, given in the table below:
\[
\begin{array}{c|cccccc}
 x & 1 & a & b & c  \\ \hline
\varphi_1(x) & 1 & a & a & a \\
\varphi_2(x) & 1 & a & b & a \\
\varphi_3(x) & 1 & a & b & c \\
\varphi_4(x) & 1 & a & c & c \\
\varphi_5(x) & a & a & a & a \\
\varphi_6(x) & b & a & b & a \\
\varphi_7(x) & b & a & b & c \\
\varphi_8(x) & c & a & c & c  
\end{array}
.   
\]
Then $\mathcal{INTO}(A)=\{\varphi_1,\cdots,\varphi_8\}$. 
\end{example}

\bigskip

\section{Very true pseudo-BCK algebras}

In this section we introduce the very true operators on pseudo-BCK algebras and we study their properties. 
We prove that the composition of two very true operators is a very true operator if and only if they commute.  
It is proved that, if two very true operators have the same image, then the two operators coincide. 
For a very true bounded pseudo-BCK algebra $(A,v)$, we define the pseudo-BCK$_{vt,st}$ algebra by adding two truth-depressing hedges operators associated with $v$. 
Given a very true operator on a good pseudo-BCK algebra $A$ satisfying Glivenko property we define  
very true operators on $\Reg(A)$ and $A/\Den(A)$. 
Finally, it is proved that the composition of a pseudo-valuation with a very true operator is a pseudo-valuation 
as well. 

\begin{definition} \label{vt-op-10} Let $A$ be a pseudo-BCK algebra. A mapping $v:A\longrightarrow A$ is called a \emph{very true operator} on $A$ if it satisfies the following conditions for all $x, y\in A:$ \\
$(VT_1)$ $v(1)=1;$ \\
$(VT_2)$ $v(x)\le x;$ \\
$(VT_3)$ $v(x)\le vv(x);$ \\
$(VT_4)$ $v(x\ra y)\le v(x)\ra v(y)$ and $v(x\rs y)\le v(x)\rs v(y)$. \\
The pair $(A,v)$ is called a \emph{very true pseudo-BCK algebra}.  
\end{definition}

Denote $\mathcal{VTO}(A)$ the set of all very true operators on $A$. \\ 
For $v\in \mathcal{VTO}(A)$, $\Ker(v)=\{x\in A\mid v(x)=1\}$ is called the \emph{kernel} of $v$. 

\begin{remark} \label{vt-op-10-10}
Axiom $(VT_1)$ means that absolutely true is very true, while $(VT_2)$ means that if $\varphi$ is very true 
then it is true. 
$(VT_3)$ says that very true of very true is very true, which is a kind of necessitation with respect to very 
true connective. 
$(VT_4)$ means that if $\varphi$, $\varphi\ra \psi$ and $\varphi\rs \psi$ are very true then so is $\psi$ 
(see the comments from \cite{Haj1}, \cite{Wang1}, \cite{Wang2}). 
\end{remark}

\begin{example} \label{vt-op-10-20} The two boundary cases of truth-stressing hedges are the following: \\
$(1)$ Identity: $Id_A\in \mathcal{VTO}(A)$ for any pseudo-BCK algebra. \\
$(2)$ Globalization (\cite{Tak1}): If $A$ is a bounded pseudo-BCK algebra and $v_g:A\longrightarrow A$ is defined by 
\begin{equation*}
v_g(x)= \left\{
\begin{array}{cc}
	1,\:\: {\rm if} \: x=1 \\
	0,\:\: {\rm if} \: x<1,
\end{array}
\right.
\end{equation*}
then $v_g\in \mathcal{VTO}(A)$. 
Globalization can be seen as an interpretation of a connective "absolutely/fully true" (\cite{Liu1}, \cite{Liu2}). 
\end{example}

\begin{proposition} \label{vt-op-20} Let $(A,v)$ be a very true pseudo-BCK algebra. Then the following 
hold for all $x, y\in A:$ \\
$(1)$ $v(x)=1$ if and only if $x=1;$ \\
$(2)$ $x\le y$ implies $v(x)\le v(y);$ \\
$(3)$ $vv(x)=v(x);$ \\
$(4)$ $v(x)\le y$ if and only if $v(x)\le v(y);$ \\
$(5)$ $\Img(v)=\Fix(v);$ \\
$(6)$ if $v$ is surjective, then $v=Id_A;$ \\
$(7)$ $\Ker(v)=\{1\};$ \\
$(8)$ $\Ker(v)\in \mathcal{DS}(A)$. 
\end{proposition}
\begin{proof}
$(1)$ Suppose that there exists $x\in A$ such that $v(x)=1$. Then by $(VT_2)$, $1=v(x)\le x$, that is $x=1$. 
Conversely, by $(VT_1)$, $v(1)=1$. \\
$(2)$ Let $x, y\in A$, $x\le y$, that is $x\ra y=x\rs y=1$.  
Applying $(VT_4)$ we have $1=v(1)=v(x\ra y)\le v(x)\ra v(y)$, hence $v(x)\ra v(y)=1$, so $v(x)\le v(y)$. \\
$(3)$ Since by $(VT_2)$, $vv(x)\le v(x)$, taking into consideration $(VT_3)$ we get $vv(x)=v(x)$. \\
$(4)$ If $v(x)\le y$, then $v(x)=vv(x)\le v(y)$. Conversely, from $v(x)\le v(y)$ and $(VT_2)$ we get 
$v(x)\le v(y)\le y$. \\
$(5)$ If $y\in \Img(v)$, there exists $x\in A$ such that $v(x)=y$. 
Then by $(3)$, $v(y)=vv(y)=vvv(x)=v(x)=y$, so $y\in \{x\in A\mid v(x)=x\}=\Fix(v)$.
It follows that $\Img(v)\subseteq \Fix(v)$. Since the converse is obvious, we get $\Img(v)=\Fix(v)$. \\
$(6)$ Let $x\in A$. Since $A=\Img(v)$, there exists $x^{\prime}\in A$ such that $x=v(x^{\prime})$. \\
Hence $v(x)=vv(x^{\prime})=v(x^{\prime})=x$. Thus $v=Id_A$. \\
$(7)$ Let $x\in \Ker(v)$, that is $v(x)=1$. It follows that $1=v(x)\le x$, that is $x=1$. \\
Hence $\Ker(v)=\{1\}$. \\
$(8)$ Let $x, y\in A$ such that $x, x\ra y\in \Ker(v)$, that is $v(x)=v(x\ra y)=1$. \\
Since $1=v(x\ra y)\le v(x)\ra v(y)=1\ra v(y)=v(y)$, we get $v(y)=1$, so $y\in \Ker(v)$. \\
Thus $\Ker(v)\in \mathcal{DS}(A)$. 
\end{proof}

Let $(A,v)$ be a very true bounded pseudo-BCK algebra. Define $\varsigma_v^1, \varsigma_v^2:A\longrightarrow A$, by 
$\varsigma_v^1(x)=v(x^{-})^{\sim}$, $\varsigma_v^2(x)=v(x^{\sim})^{-}$, for all $x\in A$. 

\begin{proposition} \label{vt-op-30-10} Let $(A,v)$ be a very true bounded pseudo-BCK algebra. Then the following 
hold for all $x, y\in A:$ \\
$(1)$ $\varsigma_v^1(0)=\varsigma_v^2(0)=0;$ \\
$(2)$ $x\le \varsigma_v^1(x)$ and $x\le \varsigma_v^2(x);$ \\
$(3)$ $x\le y$ implies $\varsigma_v^1(x)\le \varsigma_v^1(y)$ and $\varsigma_2(x)\le \varsigma_2(y);$ \\   
$(4)$ $\varsigma_v^1\varsigma_v^1(x)=\varsigma_v^1(x)$ and $\varsigma_v^2\varsigma_v^2(x)=\varsigma_v^2(x)$. 
\end{proposition}
\begin{proof}
$(1)$ It is obvious. \\
$(2)$ It follows from Proposition \ref{int-op-50}$(5)$. \\
$(3)$ From $x\le y$ we have $y^{-}\le x^{-}$, hence $v(y^{-})\le v(x^{-})$. 
It follows that $v(x^{-})^{\sim}\le v(y^{-})^{\sim}$, that is $\varsigma_v^1(x)\le \varsigma_v^1(y)$. \
Similarly $\varsigma_v^2(x)\le \varsigma_v^2(y)$. \\
$(4)$ For all $x\in A$, we have: \\
$\hspace*{2cm}$ $\varsigma_v^1\varsigma_v^1(x)=\varsigma_v^1(v(x^{-})^{\sim})=
                  (v((v(x^{-}))^{\sim})^{-})^{\sim}=(v(v(x^{-}))^{\sim-})^{\sim}$. \\
From $v(x^{-})\le (v(x^{-}))^{\sim-}$ we get $vv(x^{-})\le v((v(x^{-}))^{\sim-})$, hence 
$v(x^{-})\le v((v(x^{-}))^{\sim-})$. \\ 
It follows that: \\ 
$\hspace*{2cm}$ $\varsigma_v^1(x)=(v(x^{-})^{\sim}\ge (v((v(x^{-}))^{\sim-}))^{\sim}=\varsigma_v^1\varsigma_v^1(x)$. \\ 
On the other hand, by $(2)$, $\varsigma_v^1(x)\le \varsigma_v^1\varsigma_v^1(x)$, hence $\varsigma_v^1\varsigma_v^1(x)=\varsigma_v^1(x)$. \\
Similarly $\varsigma_v^2\varsigma_v^2(x)=\varsigma_v^2(x)$. 
\end{proof}

\begin{remark} \label{vt-op-30-20} By Proposition \ref{vt-op-30-10}$(2)$,$(3)$,$(4)$, if $(A,v)$ is a very true 
bounded pseudo-BCK algebra, then $\varsigma_v^1, \varsigma_v^2 \in \mathcal{CLO}(A)$.  
\end{remark}

\begin{example} \label{vt-op-40} Let $(A, \ra, \rs, 1)$ be the pseudo-BCK algebra from 
Example \ref{psBE-50} and the maps $v_i:A\longrightarrow A$, $i=1,2,3,4$, given in the table below:
\[
\begin{array}{c|cccccc}
 x & 1 & a & b & c  \\ \hline
v_1(x) & 1 & a & a & a \\
v_2(x) & 1 & a & b & a \\
v_3(x) & 1 & a & b & c \\
v_4(x) & 1 & a & c & c  
\end{array}
.   
\]
Then $\mathcal{VTO}(A)=\{v_1,v_2,v_3,v_4\}$. 
\end{example}

\begin{remark} \label{vt-op-50} In any pseudo-BCK algebra $A$, $\mathcal{VTO}(A)\subseteq \mathcal{INTO}(A)$. 
As we can see in Examples \ref{vt-op-40} and \ref{int-op-80}, $\mathcal{VTO}(A)\ne \mathcal{INTO}(A)$.
\end{remark}

\begin{theorem} \label{vt-op-80} Let $A$ be a pseudo-BCK algebra and let $v_1, v_2\in \mathcal{VTO}(A)$. 
If $\Img(v_1)=\Img(v_2)$, then $v_1=v_2$. 
\end{theorem}
\begin{proof}
According to Proposition \ref{vt-op-20}$(5)$ we have $\Fix(v_1)=\Fix(v_2)$ and applying Theorem \ref{int-op-30} it 
follows that $v_1=v_2$.
\end{proof}

\begin{theorem} \label{vt-op-90} Let $A$ be pseudo-BCK algebra and $v_1, v_2\in \mathcal{VTO}(A)$. 
Then $v_1 \circ v_2\in \mathcal{VTO}(A)$ if and only if $v_1\circ v_2=v_2\circ v_1$.  
\end{theorem}
\begin{proof} 
Let $v_1, v_2\in \mathcal{VTO}(A)$ such that $v_1 \circ v_2\in \mathcal{VTO}(A)$. 
According to Remark \ref{vt-op-50} we have $v_1 \circ v_2\in \mathcal{INTO}(A)$ and applying 
Theorem \ref{int-op-20} it follows that $v_1\circ v_2=v_2\circ v_1$. \\
Conversely, let $v_1, v_2\in \mathcal{VTO}(A)$ such that $v_1\circ v_2=v_2\circ v_1$ and let 
$v=v_1\circ v_2=v_2\circ v_1$. \\ 
By Theorem \ref{int-op-20}, $v\in \mathcal{INTO}(A)$, that is $v$ satisfies $(VT_2)$ and $(VT_3)$.  
Obviously $v(1)=1$, hence $v$ satisfies $(VT_1)$. 
Since $v_1$ and $v_2$ satisfy $(VT_4)$ we have: \\
$\hspace*{2cm}$ $v(x\ra y)=v_1(v_2(x\ra y))\le v_1(v_2(x)\ra v_2(y))$ \\
$\hspace*{3.6cm}$ $\le v_1(v_2(x))\ra v_1(v_2(y))=v(x)\ra v(y)$. \\ 
Similarly $v(x\rs y)\le v(x)\rs v(y)$, that is $v$ satisfies $(VT_4)$. Hence $v\in \mathcal{VTO}(A)$. 
\end{proof}

\begin{example} \label{vt-op-100} Let $(A, \ra, \rs,1)$ be the pseudo-BCK algebra from 
Example \ref{vt-op-40} with $\mathcal{VTO}(A)=\{v_1,v_2,v_3,v_4\}$. 
One can easily check that $v_1\circ v_2=v_2\circ v_1=v_1\in \mathcal{VTO}(A)$, while 
$v_4\circ v_2\ne v_2\circ v_4$ and $v_4\circ v_2 \notin \mathcal{VTO}(A)$.
\end{example}

\begin{proposition} \label{vt-op-110} Let $(A, \ra, \rs, 0, 1)$ be a good pseudo-BCK algebra satisfying 
Glivenko property and let $v\in \mathcal{VTO}(A)$. Define $\tilde{v}:\Reg(A)\longrightarrow \Reg(A)$ by $\tilde{v}(x)=v(x)^{-\sim}$. Then $\tilde{v}\in \mathcal{VTO}(\Reg(A))$.  
\end{proposition}
\begin{proof} 
According to Proposition \ref{int-op-60}, $\tilde{v}$ satisfies $(VT_2)$ and $(VT_3)$. \\
Obviously $\tilde{v}(1)=1$, hence $(VT_1)$ is also verified. \\
Since $A$ has Glivenko property and $v\in \mathcal{VTO}(A)$, we have: \\
$\hspace*{1cm}$ $\tilde{v}(x\ra y)=v(x\ra y)^{-\sim}\le (v(x)\ra v(y))^{-\sim}=v(x)^{-\sim}\ra v(y)^{-\sim}= 
\tilde{v}(x)\ra \tilde{v}(y)$. \\
Similarly $\tilde{v}(x\rs y)\le \tilde{v}(x)\rs \tilde{v}(y)$, thus $\tilde{v}$ satisfies $(VT_4)$. 
Hence $\tilde{v}\in \mathcal{VTO}(\Reg(A))$. 
\end{proof}

\begin{proposition} \label{vt-op-120} Let $(A, \ra, \rs, 0, 1)$ be a good pseudo-BCK algebra satisfying 
Glivenko property and let $v\in \mathcal{VTO}(A)$. Define $\tilde{v}:A/\Den(A)\longrightarrow A/\Den(A)$ by $\tilde{v}([x]_{\Den(A)})=[v(x)]_{\Den(A)}$. Then $\tilde{v}\in \mathcal{VTO}(A/\Den(A))$.  
\end{proposition}
\begin{proof} 
By Proposition \ref{int-op-70}, $\tilde{v}$ is well defined and it satisfies $(VT_2)$ and $(VT_3)$. \\ 
Clearly $\tilde{v}([1]_{\Den(A)})=[1]_{\Den(A)}$, that is $(VT_1)$ . 
For any $x, y\in A$ we have: \\
$\hspace*{2cm}$ $\tilde{v}([x]_{\Den(A)}\ra [y]_{\Den(A)})=\tilde{v}([x\ra y]_{\Den(A)})=[v(x\ra y)]_{\Den(A)}$ \\ 
$\hspace*{6.2cm}$ $\le [(v(x)\ra v(y))]_{\Den(A)}=[v(x)]_{\Den(A)}\ra [v(y)]_{\Den(A)}$ \\
$\hspace*{6.2cm}$ $=\tilde{v}([x]_{\Den(A)})\ra \tilde{v}([y]_{\Den(A)})$, \\
 Similarly $\tilde{v}([x]_{\Den(A)}\rs [y]_{\Den(A)})\le \tilde{v}([x]_{\Den(A)})\rs \tilde{v}([y]_{\Den(A)})$, 
hence $(VT_4)$ is satisfied. \\
We conclude that $\tilde{v}\in \mathcal{VTO}(A/\Den(A))$. 
\end{proof}

In order to study the truth-depressing hedges on BL-algebras, Vychodil (\cite{Vyc1}) introduced the $BL_{vt,st}$ algebra extending a very true BL-algebra $(A,\wedge,\vee,\odot,\ra,v,0,1)$ by an additional unary 
operator $s:A\longrightarrow A$ satisfying the following conditions, for all $x, y\in A:$ \\
$(s_1)$ $s(0)=0;$ \\
$(s_2)$ $x\le s(x);$ \\
$(s_3)$ $v(x\ra y)\le s(x)\ra s(y)$. \\
The operator $s$ is called a \emph{truth-depressing hedge associated with $v$}. \\
Similarly, we can define the pseudo-BCK$_{vt,st}$ algebra extending a very true bounded pseudo-BCK algebra 
$(A,\ra,\rs,v,0,1)$ by two additional unary operators. 

\begin{definition} \label{vt-op-130}
A pseudo-BCK$_{vt,st}$ algebra is a very true bounded pseudo-BCK algebra $(A,\ra,\rs,v,0,1)$ endowed with two 
unary operators $s_1,s_2:A\longrightarrow A$ satisfying the following conditions, for all $x, y\in A:$ \\
$(ST_1)$ $s_1(0)=s_2(0)=0;$ \\
$(ST_2)$ $x\le s_1(x)$ and $x\le s_2(x);$ \\
$(ST_3)$ $v(x\ra y)\le s_1(x)\ra s_1(y)$ and $v(x\rs y)\le s_2(x)\rs s_2(y)$. 
\end{definition}
The pseudo-BCK$_{vt,st}$ algebra will be denoted by $(A,\ra,\rs,v,s_1,s_2,0,1)$, and the pair $(s_1, s_2)$ is called a truth-depressing hedge associated with $v$.

\begin{example} \label{vt-op-140-10} Let $(A,v)$ be a very true bounded pseudo-BCK algebra. Then \\
$(A,\ra,\rs,v,Id_A,Id_A,0,1)$ is a pseudo-BCK$_{vt,st}$ algebra. 
\end{example}

The next result is a generalization of \cite[Ex. 1]{Vyc1} from BL-algebras to the case of pseudo-BCK algebras. 

\begin{proposition} \label{vt-op-140-30} Let $A$ be a linearly ordered bounded pseudo-BCK algebra and let $v_g$ be  
the globalization operator from Example \ref{vt-op-10-20}. If $s_1,s_2:A\longrightarrow A$ are monotone operators  
satisfying $(ST_1), (ST_2)$, then $(A,\ra,\rs,v_g,s_1,s_2,0,1)$ is a pseudo-BCK$_{vt,st}$ algebra.
\end{proposition}
\begin{proof}
Let $x,y\in A$. If $x\le y$, then $s_1(x)\le s_1(y)$ and $s_2(x)\le s_2(y)$, and so 
$v_g(x\ra y)=v_g(1)=1=s_1(x)\ra s_1(y)$ and $v_g(x\rs y)=v_g(1)=1=s_2(x)\rs s_2(y)$. 
If $x\nleq y$, then $v_g(x\ra y)=0\le s_1(x)\ra s_1(y)$ and $v_g(x\rs y)=0\le s_2(x)\rs s_2(y)$. 
Hence $(ST_3)$ is satisfied and $(A,\ra,\rs,v_g,s_1,s_2,0,1)$ is a pseudo-BCK$_{vt,st}$ algebra.    
\end{proof}

\begin{theorem} \label{vt-op-140}
Let $(A,v)$ be a very true bounded pseudo-BCK algebra. Define $\varsigma_v^1,\varsigma_v^2:A\longrightarrow A$, by 
$\varsigma_v^1(x)=v(x^{-})^{\sim}$, $\varsigma_v^2(x)=v(x^{\sim})^{-}$, for all $x\in A$. 
Then $(A,\ra,\rs,v,\varsigma_v^1,\varsigma_v^2,0,1)$ is a pseudo-BCK$_{vt,st}$ algebra. 
\end{theorem}
\begin{proof}
According to Proposition \ref{vt-op-30-10}, $\varsigma_v^1$, $\varsigma_v^2$ satisfy the conditions $(ST_1)$, $(ST_2)$. 
Applying twice Proposition \ref{psBE-100}$(9)$ we get: \\
$\hspace*{1.5cm}$ $v(x\ra y)\le v(y^{-}\rs x^{-})\le v(y^{-})\rs v(x^{-})\le v(x^{-})^{\sim}\ra v(y^{-})^{\sim}=
                 \varsigma_v^1(x)\ra \varsigma_v^1(y)$ \\
$\hspace*{1.5cm}$ $v(x\rs y)\le v(y^{\sim}\ra x^{\sim})\le v(y^{\sim})\ra v(x^{\sim})\le 
                 v(x^{\sim})^{-}\rs v(y^{\sim})^{-}= \varsigma_v^2(x)\rs \varsigma_v^2(y)$, \\ 
that is $(ST_3)$. Hence $(A,\ra,\rs,v,\varsigma_v^1,\varsigma_v^2,0,1)$ is a pseudo-BCK$_{vt,st}$ algebra. 
\end{proof}

\begin{proposition} \label{vt-op-140-20} Let $(A,\ra,\rs,v,s_1,s_2,0,1)$ be a pseudo-BCK$_{vt,st}$ algebra. 
Then $Id_A\le s_1\le \varsigma_v^1$ and $Id_A\le s_1\le \varsigma_v^2$. 
\end{proposition}
\begin{proof}
By $(ST_3)$ and Lemma \ref{psBE-40-20}$(2)$ we have $(s_1(x)\ra s_1(0))\rs 0\le v(x\ra 0)\rs 0$.  
Hence $x\le s_1(x)\le (s_1(x)\ra 0)\rs 0=(s_1(x)\ra s_1(0))\rs 0\le v(x\ra 0)\rs 0=\varsigma_v^1(x)$, that is 
$Id_A\le s_1\le \varsigma_v^1$. Similarly $Id_A\le s_2\le \varsigma_v^2$. 
\end{proof}

\begin{example} \label{vt-op-140-40} Let $v_g$ be the globalization operator from Example \ref{vt-op-10-20}. 
Then we have: 
\begin{equation*}
\varsigma_{v_g}^1(x)=\varsigma_{v_g}^2(x)= \left\{
\begin{array}{cc}
	0,\:\: {\rm if} \: x=0 \\
	1,\:\: {\rm if} \: x>0,
\end{array}
\right.
\end{equation*}
Indeed, $x^{-}=1$ iff $x\ra 0=1$ iff $x\le 0$ iff $x=0$ and similarly $x^{\sim}=1$ iff $x=0$. 
Hence $x^{-}, x^{\sim}<1$ and $v_g(x^{-})=v_g(x^{\sim})=0$, for all $x>0$. 
According to Theorem \ref{vt-op-140}, $(A,\ra,\rs,v,\varsigma_{v_g}^1,\varsigma_{v_g}^2,0,1)$ is a 
pseudo-BCK$_{vt,st}$ algebra. 
\end{example}

\emph{Pseudo-valuations} on a pseudo-BCK algebra $A$ were defined in \cite{Ciu2} as 
real-valued functions $\varphi:A\longrightarrow {\mathbb R}$ satisfying the conditions: \\
$(pv_1)$ $\varphi(1)=0;$ \\
$(pv_2)$ $\varphi(y)-\varphi(x)\le \min\{\varphi(x\ra y), \varphi(x\rs y)\}$  
for all $x, y\in A$. \\ 
A pseudo-valuation $\varphi$ is said to be a \emph{valuation} if it satisfies the condition: \\
$(pv_3)$ $v(x)=0$ implies $x=1$ for all $x\in A$. \\
Denote $\mathcal{PV}(A)$ the set of all pseudo-valuations on $A$. \\
If $\varphi\in \mathcal{PV}(A)$, then the following hold for all $x, y\in A$: \\
$(pv_4)$ $\varphi(x)\ge \varphi(y)$, whenever $x\le y$ ($\varphi$ is order reversing); \\
$(pv_5)$ $\varphi(x)\ge 0$. 

\begin{theorem} \label{vt-op-150}Let $A$ be a pseudo-BCK algebra and let $v\in \mathcal{VTO}(A)$. 
If $\varphi \in \mathcal{PV}(A)$, then $\varphi_v=\varphi\circ v \in \mathcal{PV}(A)$.  
\end{theorem}
\begin{proof}
Obviuosly $\varphi_v(1)=\varphi(v(1))=\varphi(1)=0$, so $\varphi_v$ satisfies $(pv_1)$. \\ 
Since by $(VT_4)$ and $(pv_4)$ we have $\varphi(v(x)\ra v(y))\le \varphi(v(x\ra y))$ and 
$\varphi(v(x)\rs v(y))\le \varphi(v(x\rs y))$, applying $(pv_2)$ we get: \\
$\hspace*{2cm}$ $\varphi_v(y)-\varphi_v(x)=\varphi(v(y))-\varphi(v(x))\le 
                 \min\{\varphi(v(x)\ra v(y)),\varphi(v(x)\rs v(y))\}$ \\ 
$\hspace*{4.5cm}$ $\le \min \{\varphi(v(x\ra y)),\varphi(v(x\rs y))\}$ \\ 
$\hspace*{4.5cm}$ $= \min \{\varphi\circ v(x\ra y),\varphi\circ v(x\rs y)\}$ \\  
$\hspace*{4.5cm}$ $=\min \{\varphi_v(x\ra y), \varphi_v(x\rs y)\}$, \\
for all $x, y\in A$, that is $(pv_2)$. Hence $\varphi_v \in \mathcal{PV}(A)$.
\end{proof}

\begin{example} \label{vt-op-160} Let $(A, \ra, \rs, a, 1)$ be the bounded pseudo-BCK algebra from 
Example \ref{psBE-50} and $v\in \mathcal{VTO}(A)$, $\varphi\in \mathcal{PV}(A)$ given in the table below: 
\[
\begin{array}{c|cccccc}
 x           & 1 & a & b & c  \\ \hline
v(x)         & 1 & a & b & a \\
\varphi(x)   & 0 & 3 & 1 & 2 \\
\varphi_v(x) & 0 & 3 & 1 & 3 
\end{array}
.   
\]
One can easily check that $\varphi_v=\varphi\circ v\in \mathcal{PV}(A)$.  
\end{example}

\bigskip

\section{Very true deductive systems and very true pseudo-BCK homomorphisms}

We define the very true deductive systems of a very true pseudo-BCK algebra $(A,v)$ and we prove that the congruences 
of $(A,v)$ coincide with the congruences of $A$. We also introduce the notion of a very true homomorphism and we 
investigate their properties. Given a normal $v$-deductive system of a very true pseudo-BCK algebra $(A,v)$ we construct  
a very true operator on the quotient pseudo-BCK algebra $A/H$.  
Given two very true pseudo-BCK algebras $(A, v)$ and $(B, u)$ and $\psi$ a very true homomorphism we prove that 
the image of a very true subalgebra of $A$ is a very true subalgebra of $B$ and the kernel of $\psi$ is a very 
true deductive system of $A$. If moreover $\psi$ is surjective, it is proved that the image of a very true 
deductive system of $A$ is a very true deductive system of $B$. 

\begin{definition} \label{vt-ds-10} Let $(A,v)$ be a very true pseudo-BCK algebra and let $D\in \mathcal{DS}(A)$. 
Then $D$ is called a \emph{$v$-deductive system} of $A$ if $v(D)\subseteq D$. 
If $D\in \mathcal{DS}_n(A)$ such that $v(D)\subseteq D$, then $D$ is called a \emph{normal $v$-deductive system} 
of $A$. 
\end{definition} 
Denote $\mathcal{DS}^v(A)$ the set of all $v$-deductive systems of $A$ and $\mathcal{DS}^v_n(A)$ the set 
of all normal $v$-deductive systems of $A$.

\begin{remark} \label{vt-ds-20} $\{1\}, A, \Ker(v)\in$ $\mathcal{DS}^v(A)$, for all $v\in \mathcal{VTO}(A)$. 
\end{remark}

\begin{example} \label{vt-ds-30} Let $(A, \ra, \rs, 1)$ be the pseudo-BCK algebra from 
Example \ref{vt-op-40} with $\mathcal{VTO}(A)=\{v_1,v_2,v_3,v_4\}$. 
We have $\mathcal{DS}(A)=\{\{1\},\{1,b\},A\}$. One can easily see that 
$\mathcal{DS}^{v_1}(A)=\mathcal{DS}^{v_4}(A)=\{\{1\},A\}$ and 
$\mathcal{DS}^{v_2}(A)=\mathcal{DS}^{v_3}(A)=\mathcal{DS}(A)$. 
\end{example}

\begin{theorem} \label{vt-ds-40} Let $(A, v)$ be a very true pseudo-BCK algebra and let 
$H\in \mathcal{DS}^v_n(A)$. 
Define $\hat{v}:A/H \longrightarrow A/H$ by $\hat{v}([x]_H)=[v(x)]_H$, for all $x\in A$. 
Then $\hat{v}\in \mathcal{VTO}(A/H)$. 
\end{theorem}
\begin{proof} If $[x]_H=[y]_H$, then $x\ra y, y\ra x\in H$, hence $v(x\ra x), v(y\ra x)\in H$. 
By $(VT_4)$ we get $v(x)\ra v(y), v(y)\ra v(x)\in H$, that is $[v(x)]_H=[v(y)]_H$, hence $\hat{v}$ is well defined. \\
Clearly $\hat{v}([1]_H)=[v(1)]_H=[1]_H$ and $\hat{v}([x]_H)=[v(x)]_H\le [x]_H$, 
hence $\hat{v}$ satisfies $(VT_1)$ and $(VT_2)$. For all $x\in A$ we have: \\
$\hspace*{2cm}$ $\hat{v}\hat{v}([x]_H)=\hat{v}([v(x)]_H)=[vv(x)]_H=[v(x)]_H=\hat{v}([x]_H)$, \\
so $(VT_3)$ is verified. Finaly, for all $x, y\in A$ we have: \\
$\hspace*{2cm}$ $\hat{v}([x]_H\ra [y]_H)=\hat{v}([x\ra y]_H)=[v(x\ra y)]_H \le [v(x)\ra v(y)]_H $ \\
$\hspace*{4.6cm}$ $=[v(x)]_H\ra [v([y]_H=\hat{v}([x]_H)\ra \hat{v}([y]_H)$. \\
Similarly $\hat{v}([x]_H\rs [y]_H) \le \hat{v}([x]_H)\rs \hat{v}([y]_H$, thus $(VT_4)$ holds. 
Hence $\hat{v}\in \mathcal{VTO}(A/H)$. 
\end{proof}

\begin{proposition} \label{vt-ds-50} The congruences on a very true pseudo-BCK algebra $(A, v)$ 
coincide with the congruences of $A$.
\end{proposition}
\begin{proof} 
Let $\theta\in \mathcal{CON}(A)$. 
There is a one-to-one correspondence between $\mathcal{CON}(A)$ and $\mathcal{DS}_n(A)$. \\
Let $H\in \mathcal{DS}_n(A)$ and let $(x,y)\in \theta_H$, that is $x\ra y, y\ra x\in H$. 
It follows that $v(x\ra y), v(y\ra x)\in H$. Since $v(x\ra y)\le v(x)\ra v(y)$ and $v(y\ra x)\le v(y)\ra v(x)$, 
we get $v(x)\ra v(y), v(y)\ra v(x)\in H$, that is $(v(x), v(y))\in \theta_H$. 
Hence $\theta_H\in \mathcal{CON}(A,v)$. 
\end{proof}

\begin{definition} \label{vt-ds-60} Let $(A, v)$ and $(B, u)$ be very true pseudo-BCK algebras and let 
$\psi:A\longrightarrow B$ be a pseudo-BCK homomorphism.  
Then $\psi$ is called a \emph{very true pseudo-BCK homomorphism} if $\psi(v(x))=u(\psi(x))$, for all $x\in A$.
\end{definition} 

Denote $\mathcal{VHOM}((A,v), (B,u))$ the sets of all very pseudo-BCK homomorphisms from $(A,v)$ to $(B,u)$. 
For $\psi\in \mathcal{VHOM}((A,v), (B,u))$, $\Ker(\psi)=\{x\in A\mid \psi(x)=1\}$ is called the 
\emph{kernel} of $\psi$. 

\begin{remark} \label{vt-ds-70} If $(A, v)$ is a very true pseudo-BCK algebra, then 
$1_A, Id_A\in \mathcal{VHOM}((A,v), (A,v))$. 
\end{remark} 

\begin{example} \label{vt-ds-80}
Let $(A, \ra, \rs, e, 1)$ be the bounded pseudo-BCK algebra from Example \ref{psBE-50-10-10} and define the 
maps $v_i:A\longrightarrow A$, $i=1,\cdots,10$ and $\psi_i:A\longrightarrow A$, $i=1,2,3$ as in the tables below:
\[
\begin{array}{c|cccccc}
 x & 1 & a & b & c & d & e \\ \hline
v_1(x) & 1 & a & b & c & d & e \\
v_2(x) & 1 & a & e & a & a & e \\
v_3(x) & 1 & a & e & a & d & e \\ 
v_4(x) & 1 & a & e & c & a & e \\ 
v_5(x) & 1 & e & b & b & b & e \\ 
v_6(x) & 1 & e & b & b & d & e \\ 
v_7(x) & 1 & e & b & c & b & e \\ 
v_8(x) & 1 & e & e & c & e & e \\
v_9(x) & 1 & e & e & e & d & e \\
v_{10}(x) & 1 & e & e & e & e & e
\end{array} 
\]
\[
\begin{array}{c|cccccc}
 x & 1 & a & b & c & d & e \\ \hline
\psi_1(x) & 1 & 1 & 1 & 1 & 1 & 1 \\
\psi_2(x) & 1 & a & b & c & d & e \\
\psi_3(x) & 1 & b & a & d & c & e 
\end{array}
.   
\]
Then: \\
$(1)$ $VTO(A)=\{v_1,\cdots,v_{10}\};$ \\ 
$(2)$ $\mathcal{HOM}(A,A)=\{\psi_1,\psi_2,\psi_3\};$ \\
$(3)$ $\mathcal{VHOM}((A,v_i), (A,v_i))=\{\psi_1, \psi_2\}$, $i=1,\cdots,9;$ \\ 
$(4)$ $\mathcal{VHOM}((A,v_{10}), (A,v_{10}))=\{\psi_1, \psi_2, \psi_3\}$. 
\end{example}

\begin{theorem} \label{vt-ds-90} Let $(A, v)$ and $(B, u)$ be very true pseudo-BCK algebras and let \\
$\psi\in \mathcal{VHOM}((A,v), (B,u))$. Then the following hold: \\
$(1)$ if $(A^{\prime},v)$ is a very true subalgebra of $A$, then $(\psi(A^{\prime}),u)$ is a very true 
subalgebra of $B;$ \\
$(2)$ $\Ker(\psi)\in \mathcal{DS}^v_n(A);$ \\
$(3)$ if $\psi$ is surjective and $D\in \mathcal{DS}^v(A)$, then $\psi(D)\in \mathcal{DS}^u(B);$ \\
$(4)$ if $G\in \mathcal{DS}^u(B)$, then $\psi^{-1}(G)\in \mathcal{DS}^v(A)$. 
\end{theorem} 
\begin{proof} 
$(1)$ Let $A$ be a subalgebra of $A$. It is known that $\psi(A^{\prime})$ is a subalgebra of $B$, so 
we prove that $u$ satisfies axioms $(VT_1)-(VT_4)$: \\
$(VT_1):$ $u(1)=u(\psi(1))=\psi(v(1))=\psi(1)=1$. \\
$(VT_2):$ Let $y\in \psi(A^{\prime})$, so there exists $x\in A^{\prime}$ such that $y=\psi(x)$. Then we have: \\
$\hspace*{2cm}$ $u(y)=u(\psi(x))=\psi(v(x))\le \psi(x)=y$. \\
$(VT_3):$ $uu(y)=uu(\psi(x))=u(\psi(v(x)))=\psi(vv(x))=\psi(v(x))=u(\psi(x)=u(y)$. \\
$(VT_4):$ Let $y_1, y_2 \in \psi(A^{\prime})$ and let $x_1, x_2\in A^{\prime}$ such that $y_1=\psi(x_1)$ 
and $y_2=\psi(x_2)$. \\
$\hspace*{2cm}$ $u(y_1\ra y_2)=u(\psi(x_1)\ra \psi(x_2))=u(\psi(x_1\ra x_2))$ \\
$\hspace*{3.9cm}$ $=\psi(v(x_1\ra x_2))\le \psi(v(x_1)\ra v(x_2))$ \\
$\hspace*{3.9cm}$ $=\psi(v(x_1))\ra \psi(v(x_2))=u(\psi(x_1))\ra u(\psi(x_2))$ \\
$\hspace*{3.9cm}$ $=u(y_1)\ra u(y_2)$. \\
Similarly $u(y_1\rs y_2)\le u(y_1)\rs u(y_2)$. Hence $(\psi(A^{\prime}),u)$ is a very true subalgebra of $B$. \\
$(2)$ Clearly $\Ker(\psi)\in \mathcal{DS}(A)$. 
For any $x\in \Ker(\psi)$, we have $\psi(v(x))=u(\psi(x))=u(1)=1$, that is $v(x)\in \Ker(\psi)$. 
Hence $\Ker(\psi)\in \mathcal{DS}^v(A)$. Moreover: \\
$\hspace*{2cm}$ $x\ra y\in \Ker(\psi)$ iff $\psi(x\ra y)=1$ iff $\psi(x)\ra \psi(y)=1$ \\
$\hspace*{5.3cm}$ iff $\psi(x)\le \psi(y)$ iff $\psi(x)\rs \psi(y)=1$ iff $\psi(x\rs y)=1$ \\
$\hspace*{5.3cm}$ iff $x\rs y\in \Ker(\psi)$. \\
It follows that $\Ker(\psi)\in \mathcal{DS}^v_n(A)$. \\ 
$(3)$ Let $D\in \mathcal{DS}^v(A)$. One can easily check that, if $\psi$ is surjective then 
$\psi(D)\in \mathcal{DS}(B)$. \\
Let $y\in \psi(D)$ and $x\in D$ such that $\psi(x)=y$. By hypothesis, $v(x)\in D$, so $\psi(v(x))\in \psi(D)$. \\
It follows that $u(y)=u(\psi(x))=\psi(v(x))\in \psi(D)$. Thus $\psi(D)\in \mathcal{DS}^u(B)$. \\ 
$(4)$ Let $G\in \mathcal{DS}^u(B)$. It is not difficult to prove that $\psi^{-1}(G)\in \mathcal{DS}(A)$. 
Consider $x\in \psi^{-1}(G)$, so $\psi(x)\in G$. It follows that $u(\psi(x))\in G$, hence 
$\psi(v(x))\in G$, that is $v(x)\in \psi^{-1}(G)$. \\
We conclude that $\psi^{-1}(G)\in \mathcal{DS}^v(A)$. 
\end{proof}

\begin{corollary} \label{vt-ds-90-10} Let $(A, v)$ and $(B, u)$ be very true pseudo-BCK algebras and let \\
$\psi\in \mathcal{VHOM}((A,v), (B,u))$. Then the following hold: \\
$(1)$ $\psi^{-1}(\Ker(u))\in \mathcal{DS}^v(A);$ \\
$(2)$ if $\psi$ is surjective, then $\psi(\Ker(v))\in \mathcal{DS}^u(B).$ 
\end{corollary} 

\begin{proposition} \label{vt-ds-100} Let $(A, v)$ be a very true pseudo-BCK algebra and let 
$H\in \mathcal{DS}^v_n(A)$. \\
Define $\pi:A\longrightarrow A/H$ by $\pi(x)=[x]_H$ and $\hat{v}:A/H \longrightarrow A/H$ by 
$\hat{v}([x]_H)=[v(x)]/H$, for all $x\in A$. 
Then the following hold: \\
$(1)$ $\Ker(\pi)=H;$ \\
$(2)$ $\pi\in \mathcal{VHOM}((A,v),(A/H,\hat{v}));$ \\ 
$(3)$ $\pi^{-1}(\Ker(\hat{v})) \subseteq v^{-1}(H);$ \\
$(4)$ $\pi(\Ker(v))\subseteq \Ker(\hat{v})$.
\end{proposition} 
\begin{proof} 
$(1)$ It is obvious. \\
$(2)$ According to Theorem \ref{vt-ds-40}, $\hat{v}\in \mathcal{VTO}(A/H)$. 
Obviously $\pi\in \mathcal{HOM}(A, A/H)$ and we have: \\
$\hspace*{2cm}$ $\pi(v(x))=[v(x)]_H=\hat{v}([x]_H)=\hat{v}(\pi(x))$, \\ 
for all $x\in A$. Hence $\pi\in \mathcal{VHOM}((A,v),(A/H,\hat{v}))$. \\
$(3)$ Let $x\in \pi^{-1}(\Ker(\hat{v}))$, that is $\pi(x)\in \Ker(\hat{v})$. Then we have: \\
$\hspace*{2cm}$ $\pi(v(x))=\hat{v}(\pi(x))=\hat{v}(1)=[1]_H$, \\ 
hence $v(x)\in \Ker(\pi)=H$. It follows that $x\in v^{-1}(H)$, thus $\pi^{-1}(\Ker(\hat{v})) \subseteq v^{-1}(H)$. \\ 
$(4)$ If $y\in \pi(\Ker(v))$, then there exists $x\in \Ker(v)$ such that $\pi(x)=y$. We have: \\
$\hspace*{2cm}$ $\hat{v}(y)=\hat{v}(\pi(x))=\pi(v(x))=\pi(1)=[1]_H$, \\ 
hence $y\in \Ker(\hat{v})$. We conclude that $\pi(\Ker(v))\subseteq \Ker(\hat{v})$.
\end{proof}

\begin{theorem} \label{vt-ds-110} Let $(A, v)$ and $(B, u)$ be very true pseudo-BCK algebras, and let \\ 
$\psi\in \mathcal{VHOM}((A,v), (B,u))$. Consider $H\in \mathcal{DS}^v_n(A)$ such that $H\subseteq \Ker(\psi)$, 
and let \\ 
$\pi\in \mathcal{VHOM}((A,v),(A/H,\hat{v}))$, $\hat{v}\in \mathcal{VTO}(A/H)$, defined in Proposition \ref{vt-ds-100}. Then there exits a unique $\tilde{\psi}\in \mathcal{VHOM}((A/H,\hat{v}), (B,u))$ such that the following diagram is commutative.  

\begin{center}
\begin{picture}(105,100)(0,0)
\put(-25,80){$(A,v)$}
\put(8,84){\vector(1,0){88}}
\put(100,80){$(B,u)$}
\put(40,90){$\psi$}
\put(-30,2){$(A/H,\hat{v})$}
\put(-10,75){\vector(0,-1){60}}
\put(-20,45){$\pi$}
\put(0,18){\vector(2,1){107}}
\put(73,40){$\tilde{\psi}$}
\end{picture}
\end{center}
Moreover $\Img(\tilde{\psi})=\Img(\psi)$ and $\Ker(\tilde{\psi})=\Ker(\psi)/H$.
\end{theorem}
\begin{proof} 
Define $\tilde{\psi}:A/H\longrightarrow B$, by $\tilde{\psi}([x]_H)=\psi(x)$. 
Let $x, y\in A$ such that $y\in [x]_H$.  It follows that $x\ra y, y\ra x \in H\subseteq \Ker{\psi}$, 
that is $1=\psi(x\ra y)=\psi(x)\ra \psi(y)$ and $1=\psi(y\ra x)=\psi(y)\ra \psi(x)$. 
Thus $\psi(x)\le \psi(y)$ and $\psi(y)\le \psi(x)$, that is $\psi(x)=\psi(y)$. \\ 
Hence $\tilde{\psi}([x]_H)=\tilde{\psi}([y]_H)$, so that $\tilde{\psi}$ is well defined on $A/H$. \\
Since $\psi\in \mathcal{VHOM}((A,v), (B,u))$, we have: \\
$\hspace*{2cm}$ $\tilde{\psi}(\hat{v}([x]_H))=\tilde{\psi}([v(x)]_H)=\psi(v(x))=u(\psi(x))=u(\tilde{\psi}([x]_H))$, \\
for all $x\in A$. Hence $\tilde{\psi}\in \mathcal{VHOM}((A/H,\hat{v}), (B,u))$ and obviously 
$\tilde{\psi}\circ \pi=\psi$. \\
Suppose that there exists $\mu \in \mathcal{VHOM}((A/H,\hat{v}), (B,u))$ with $\tilde{\psi}\circ \pi=\mu \circ \pi$. 
It follows that $\tilde{\psi}(\pi(x))=\mu(\pi(x))$, for all $x\in A$. 
Since $\pi$ is surjective, for any element $y\in A/H$ there exists $x\in A$ such that $y=\pi(x)$, hence $\mu=\tilde{\psi}$. \\  
We can see that: \\
$\hspace*{2cm}$ $\Img(\tilde{\psi})=\{y\in B\mid$ there exists $[x]_H\in A/H, \tilde{\psi}([x]_H)=y\}$ \\ 
$\hspace*{3.2cm}$ $=\{y\in B\mid$ there exists $x\in A, \psi(x)=y\}$ \\
$\hspace*{3.2cm}$ $=\Img(\psi)$. \\
Moreover, for all $x\in A$ we have: 
$[x]/H\in \Ker(\tilde{\psi})$ if and only if $\tilde{\psi}([x]_H)=1$ if and only if 
$\psi(x)=1$ if and only if $x\in \Ker(\psi)$. 
Hence: \\
$\hspace*{2cm}$ $\Ker(\tilde{\psi})=\{[x]_H\in A/H\mid x\in \Ker(\psi)\}=\Ker(\psi)/H$. 
\end{proof}

\begin{corollary} \label{vt-ds-120} Let $(A, v)$ and $(B, u)$ be very true pseudo-BCK algebras, and let \\ 
$\psi\in \mathcal{VHOM}((A,v), (B,u))$. Then there exists a very true pseudo-BCK isomorphism between 
$(A/\Ker(\psi),v)$ and $(\Img(\psi),u)$.
\end{corollary}
\begin{proof}
Taking $H=\Ker(\psi)$ and $B=\Img(\psi)$ in Theorem \ref{vt-ds-110}, it follows that $\tilde{\psi}$ 
is a very true isomorphism between $(A/\Ker(\psi),v)$ and $(\Img(\psi),u)$. 
\end{proof}

\bigskip

\section{Very true operators on classes of pseudo-BCK algebras}

In this section we investigate the very true operators on some classes of pseudo-BCK algebras such as 
pseudo-BCK(pP) algebras, FL$_w$-algebras, pseudo-MTL algebras. 
We define the $Q$-Smarandache pseudo-BCK algebras and we introduce the notion of a very true operator on 
$Q$-Smarandache pseudo-BCK algebras. Given a $Q$-Smarandache pseudo-BCK algebra $A$ and a very true operator $v$  
on $A$, we prove that the restriction of $v$ to $Q$ is a very true operator on $Q$. 
Let us denote:
\begin{tabbing}
xxxxxxxxxxx \= --- \= xxxxxxxxxxxxxxxxxxxxxxxxxxxxxxxxxx\kill
$\mathcal{PSBCK}$ \> --- \> the class of pseudo-BCK algebras \\
$\emph{ps}\mathcal{BCK}(\mathrm{pP})$ \> --- \> the class of pseudo-BCK(pP) algebras \\ 
$\mathcal{FL}_w$ \> --- \> the class of FL$_w$-algebras $\cong$ the class of bounded pseudo-BCK(pP) algebras \\ 
$\mathrm{\emph{ps}}\mathcal{MTL}$ \> --- \> the class of pseudo-MTL algebras\\
$\mathrm{\emph{div}}\mathcal{RL}$ \> --- \> the class of divisible residuated lattices (R$\ell$-monoids) \\
$\mathrm{\emph{ps}}\mathcal{BL}$ \> --- \> the class of pseudo-BL algebras \\
$\mathrm{\emph{ps}}\mathcal{MV}$ \> --- \> the class of pseudo-MV algebras 
\end{tabbing}
The relationship between these structures can be represented as follows:
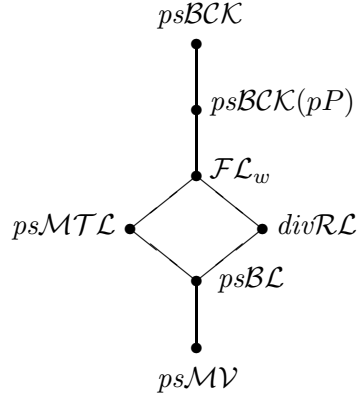
\[
\qquad\quad
\drawline(25,45)(0,25)(25,5)(50,25)(25,45)(25,70)
\drawline(25,5)(25,-18)
\drawline(25,95)(25,70)
\multiput(0,25)(25,20){2}{\circle*{4}}
\multiput(25,5)(25,20){2}{\circle*{4}}
\put(25,70){\circle*{4}}
\put(25,-20){\circle*{4}}
\put(10,-35){$\mathrm{\emph{ps}}\mathcal{MV}$}
\put(33,3){$\mathrm{\emph{ps}}\mathcal{BL}$}
\put(30,44){$\mathcal{FL}_w$} 
\put(-45,23){$\mathrm{\emph{ps}}\mathcal{MTL}$}
\put(55,23){$\mathrm{\emph{div}}\mathcal{RL}$}
\put(30,70){$\mathrm{\emph{ps}}\mathcal{BCK}(pP)$}
\put(25,95){\circle*{4}}
\put(10,103){$\mathrm{\emph{ps}}\mathcal{BCK}$}
\]

\begin{proposition} \label{vt-bck-pp-10} Let $(A, v)$ be a very true pseudo-BCK(pP) algebra. Then the following 
hold for all $x, y, z\in A:$ \\
$(1)$ if $x\odot y\le z$ then $v(x)\odot v(y)\le v(z);$ \\
$(2)$ $v(x)\odot v(y)\le v(x\odot y);$ \\
$(3)$ $v(x\ra y)\le v(x)\ra (v(z)\ra v(y))$ and $v(x\rs y)\le v(x)\rs (v(z)\rs v(y))$. 
\end{proposition}
\begin{proof}
$(1)$ From $x\odot y\le z$ we have $x\le y\ra z$, so $v(x)\le v(y)\ra v(z)$, that is $v(x)\odot v(y)\le v(z)$. \\
$(2)$ It follows from $(1)$ for $z=x\odot y$. \\
$(3)$ Since by Lemma \ref{psBCK-pP-10}, $x\ra y\le x\ra (y\ra z)$ and $x\rs y\le x\rs (y\rs z)$, $x, y, z\in A$, applying $(VT_4)$ and Proposition \ref{psBE-40-20}$(3)$ we get: \\
$\hspace*{2cm}$ $v(x\ra y)\le v(x)\ra v(z\ra y)\le v(x)\ra (v(z)\ra v(y))$, \\
$\hspace*{2cm}$ $v(x\rs y)\le v(x)\rs v(z\ra y)\le v(x)\rs (v(z)\rs v(y))$. 
\end{proof}

\begin{proposition} \label{vt-bck-pp-10-20} In any very true pseudo-BCK(pP) algebra $(A,v)$, condition $(VT_4)$ is 
equivalent to condition: \\
$(VT_4{'})$ $v(x\ra y)\le v(x)\ra (v(z)\ra v(y))$ and $v(x\rs y)\le v(x)\rs (v(z)\rs v(y))$, \\ 
for all $x, y, z\in A$.  
\end{proposition}
\begin{proof}
Suppose that condition $(VT_4)$ holds and let $x, y\in A$, hence by Proposition \ref{vt-bck-pp-10}$(3)$ 
condition $(VT_4{'})$ is satisfied. 
Conversely, if $(VT_4{'})$ is satisfied, then taking $z:=1$ we get $(VT_4)$. 
\end{proof}

\begin{proposition} \label{vt-bck-pp-10-10} In any very true pseudo-BCK(pP) algebra $(A,v)$, condition $(VT_4)$ is 
equivalent to condition: \\
$(VT_4{''})$ $v(x)\odot v(y)\le v(x\odot y)$, for all $x, y\in A$.  
\end{proposition}
\begin{proof}
Suppose that condition $(VT_4)$ holds and let $x, y\in A$. From $x\odot y\le x\odot y$ we have 
$y\le x\rs x\odot y$ and applying $(VT_4)$ we get $v(y)\le v(x\rs x\odot y)\le v(x)\rs v(x\odot y)$. \\
Hence $v(x)\odot v(y)\le v(x\odot y)$. \\
Conversely, if $(VT_4{''})$ is satisfied, from $x\ra y\le x\ra y$ we have $(x\ra y)\odot x\le y$, thus 
$v(x\ra y)\odot v(x)\le v((x\ra y)\odot x)\le v(y)$, that is $v(x\ra y)\le v(x)\ra v(y)$. \\
Similarly $v(x\rs y)\le v(x)\rs v(y)$.  
\end{proof}

\begin{definition} \label{vt-bck-pp-20} Let $(A,\wedge,\vee,\odot,\ra,\rs,0,1)$ be a FL$_w$-algebra. 
A mapping $v:A\longrightarrow A$ is called a \emph{very true operator} on $A$ if it satisfies condtitions 
$(VT_1)-(VT_4)$ and the following condition, for all $x, y\in A:$ \\
$(VT_5)$ $v(x\vee y)\le v(x)\vee v(y)$.   
\end{definition}

\begin{proposition} \label{vt-bck-pp-30} Let $(A, v)$ be a very true FL$_w$-algebra. Then the following 
hold for all $x, y\in A:$ \\
$(1)$ $v(x\ra y)\odot v(x)\le v(x\wedge y)\le v(x)\wedge v(y)$ and 
$v(x)\odot v(x\rs y)\le v(x\wedge y)\le v(x)\wedge v(y);$ \\
$(2)$ $v(x\vee y)=v(x)\vee v(y)$.    
\end{proposition}
\begin{proof}
$(1)$ Since $(x\ra y)\odot x\le x\wedge y$, applying Proposition \ref{vt-bck-pp-10}$(2)$ we have: \\
$\hspace*{2cm}$ $v(x\ra y)\odot v(x)\le v((x\ra y)\odot x)\le v(x\wedge y)\le v(x)\wedge v(y)$. \\
Similarly $v(x)\odot v(x\rs y)\le v(x\wedge y)\le v(x)\wedge v(y)$. \\
$(2)$ From $v(x), v(y)\le v(x\vee y)$ we get $v(x)\vee v(y)\le v(x\vee y)$. 
Applying $(VT_5)$ it follows that $v(x\vee y)=v(x)\vee v(y)$. 
\end{proof}

\begin{proposition} \label{vt-bck-pp-40} In any very true pseudo-MTL algebra $(A,v)$, condition $(VT_5)$ is 
equivalent to condition: \\
$(VT_5{'})$ $v(x\ra y)\vee v(y\ra x)=1$ and $v(x\rs y)\vee v(y\rs x)=1$, for all $x, y\in A$.  
\end{proposition}
\begin{proof} 
Suppose that condition $(VT_5{'})$ holds and let $x, y\in A$. 
Applying Lemmas \ref{psBCK-pP-10}$(5)$,$(4)$ and \ref{psBE-40-20}$(3)$, we get: \\
$\hspace*{2cm}$ $v(x\vee y)\le (v(x\ra y)\rs v(y))\wedge (v(y\ra x)\rs v(x))$ \\
$\hspace*{3.4cm}$ $\le (v(x\ra y)\rs (v(x)\vee v(y)))\wedge (v(y\ra x)\rs (v(x)\vee v(y)))$ \\
$\hspace*{3.4cm}$ $=(v(x\ra y)\vee v(y\ra x))\ra (v(x)\vee v(y))$ \\
$\hspace*{3.4cm}$ $=1\rs (v(x)\vee v(y))=v(x)\vee v(y)$. \\
Hence $v(x\vee y)=v(x)\vee v(y)$, that is $(VT_5)$. \\
Conversely, if $(VT_5)$ is satisfied, then: \\
$\hspace*{2cm}$ $v(x\ra y)\vee v(y\ra x)=v((x\ra y)\vee (y\ra x))=v(1)=1$ and \\ 
$\hspace*{2cm}$ $v(x\rs y)\vee v(y\rs x)=v((x\rs y)\vee (y\rs x))=v(1)=1$, \\
hence $(VT_5^{'})$ is verified. 
\end{proof}

\begin{theorem} \label{vt-bck-pp-50} A very true FL$_w$-algebra $(A,v)$ satisfies $(VT_5{'})$ if and only if 
$A$ is a pseudo-MTL algebra.
\end{theorem}
\begin{proof}
Let $(A,v)$ be a very true FL$_w$-algebra satisfying $(VT_5^{'})$ and suppose that $A$ is not a pseudo-MTL algebra. 
It follows that there exists $x, y\in A$ such that $(x\ra y)\vee(y\ra x)<1$ or $(x\rs y)\vee(y\rs x)<1$.  
If $(x\ra y)\vee(y\ra x)<1$, then $v(x\ra y)\vee v(y\ra x)\le(x\ra y)\vee (y\ra x)<1$, which is in contradiction to  
$(VT_5^{'})$. Similarly for $(x\rs y)\vee (y\rs x)<1$, hence $A$ is a pseudo-MTL algebra. 
Conversely, if $(A,v)$ is a very true pseudo-MTL algebra, then according to Proposition \ref{vt-bck-pp-40}, $(VT_5^{'})$ holds. 
\end{proof} 

\begin{theorem} \label{vt-bck-pp-60} Let $(A,\wedge,\vee,\odot,\ra,\rs,0,1)$ be a FL$_w$-algebra. 
Then the following are equivalent: \\
$(a)$ $A$ is a pseudo-MV algebra; \\
$(b)$ for any $v\in \mathcal{VTO}(A)$ the following identities hold: \\
$\hspace*{2cm}$ $v(x\vee y)=(v(x)\ra v(y))\rs v(y)=(v(x)\rs v(y))\ra v(y)$, \\
for all $x, y\in A$.
\end{theorem}
\begin{proof} 
$(a)\Rightarrow (b)$ If $A$ is a pseudo-MV algebra, then we have $x\vee y=(x\ra y)\rs y=(x\rs y)\ra y$, 
for all $x, y\in A$. Applying Proposition \ref{vt-bck-pp-30}$(2)$ we get 
$v(x\vee y)=v(x)\vee v(y)=(v(x)\ra v(y))\rs v(y)=(v(x)\rs v(y))\ra v(y)$. \\
$(b)\Rightarrow (a)$ Suppose that the identities: \\
$\hspace*{2cm}$ $v(x\vee y)=(v(x)\ra v(y))\rs v(y)=(v(x)\rs v(y))\ra v(y)$ \\ 
are satisfied for any $v\in \mathcal{VTO}(A)$. Taking $v:=Id_A$ we get $x\vee y=(x\ra y)\rs y=(x\rs y)\ra y$, 
for all $x, y\in A$. Hence $A$ is a pseudo-MV algebra.
\end{proof}

Generally, in any human field, a \emph{Smarandache structure} on a set $A$ means a weak structure $W$ on $A$ 
such that there exists a proper subset $B$ of $A$ which is embedded with a strong structure $S$ (see \cite{Jun1}). \\
If $A$ is a set endowed with a structure $W$ of a pseudo-BCK algebra, then $B$ is a subset of $A$ endowed with 
a structure $S$ which can be any of the above mentioned structures: psBCK(pP), FL$_w$, psMTL, divRL, psBL or psMV-algebra.
In this section we will consider a subset $B$ of $A$ endowed with a structure of a pseudo-MTL algebra. 

\begin{definition} \label{vt-qsm-10}
A bounded pseudo-BCK algebra $(A,\ra,\rs,0,1)$ is said to be a \emph{Q-Smarandache pseudo-BCK algebra} if there is a 
proper subset $Q$ of $A$ satisfying the following conditions: \\ 
$(S_1)$ $0, 1\in Q$ and $|Q|\geq 3;$ \\
$(S_2)$ $(Q,\wedge,\vee,\odot,\ra,\rs,0,1)$ is a pseudo-MTL algebra. 
\end{definition}

\begin{definition} \label{vt-qsm-30} Let $A$ be a $Q$-Smarandache pseudo-BCK algebra.  
A mapping $v:A \longrightarrow A$ is called a \emph{very true $Q$-Smarandache operator} if 
$v_{s}=v_{\mid Q}:Q \longrightarrow Q$ is a very true pseudo-MTL operator. 
The pair $(A,v)$ is called a \emph{very true $Q$-Smarandache pseudo-BCK algebra}.  
\end{definition}

Denote $\mathcal{SVTO}_Q(A)$ the set of all very true $Q$-Smarandache operators on a $Q$-Smarandache 
pseudo-BCK algebra $A$. 

\begin{proposition} \label{vt-qsm-40} Let $A$ be a linearly ordered $Q$-Smarandache pseudo-BCK algebra.
If $v\in \mathcal{VTO}(A)$, then $v_{\mid Q} \in \mathcal{SVTO}_Q(A)$. 
\end{proposition}
\begin{proof}
Let $v\in \mathcal{VTO}(A)$. 
It is well known that a linearly ordered $(A,\le)$ is a lattice. Moreover, according to \cite{Ior1}, 
every linearly ordered pseudo-BCK algebra is a pseudo-BCK(pP) algebra satisfying the prelinearity condition. 
Hence $v$ satisfies condition $(VT^{'}_5)$, that is $v_{\mid Q} \in \mathcal{SVTO}_Q(A)$. 
\end{proof}

\begin{proposition} \label{vt-qsm-40-10} Let $Q_1$ and $Q_2$ be pseudo-MTL algebras such that 
$Q_1\subseteq Q_2\subseteq A$. Then $\mathcal{SVTO}_{Q_2}(A) \subseteq \mathcal{SVTO}_{Q_1}(A)$. 
\end{proposition}
\begin{proof} It is straightforward. 
\end{proof}

\begin{example} \label{vt-qsm-50} Consider the set $A=\{0,a,b,c,d,1\}$ and the operations $\ra,\rs$ given by the following tables:
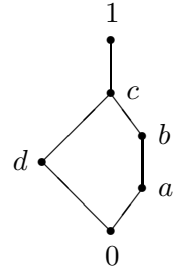
\[
\hspace{10mm}
\begin{array}{c|ccccccc}
\ra & 0 & a & b & c & d & 1 \\ \hline
0 & 1 & 1 & 1 & 1 & 1 & 1 \\ 
a & 0 & 1 & 1 & 1 & c & 1 \\ 
b & 0 & b & 1 & 1 & c & 1 \\ 
c & 0 & b & b & 1 & c & 1 \\ 
d & 0 & b & b & 1 & 1 & 1 \\
1 & 0 & a & b & c & d & 1 
\end{array}
\hspace{10mm} 
\begin{array}{c|ccccccc}
\rs & 0 & a & b & c & d & 1 \\ \hline
0 & 1 & 1 & 1 & 1 & 1 & 1 \\ 
a & 0 & 1 & 1 & 1 & c & 1 \\ 
b & 0 & c & 1 & 1 & c & 1 \\ 
c & 0 & a & b & 1 & c & 1 \\ 
d & 0 & a & b & 1 & 1 & 1 \\
1 & 0 & a & b & c & d & 1 
\end{array}
\qquad\quad
\begin{picture}(50,-70)(0,45)
\put(37,11){\circle*{3}}
\put(35,-2){$0$}
\put(37,11){\line(3,4){12}}
\put(49,27){\circle*{3}}
\put(55,24){$a$}

\put(49,27){\line(0,1){20}}

\put(49,47){\circle*{3}}
\put(55,44){$b$}

\put(49,47){\line(-3,4){12}}
\put(37,63){\circle*{3}}
\put(43,61){$c$}
\put(37,11){\line(-1,1){26}}
\put(11,37){\circle*{3}}
\put(0,34){$d$}

\put(11,37){\line(1,1){26}}

\put(37,63){\line(0,1){20}}
\put(37,83){\circle*{3}}
\put(35,90){$1$}
\end{picture}
\]
Then $(A,\rightarrow,\rightsquigarrow,0,1)$ is a bounded pseudo-BCK algebra (see \cite[Ex. 3.1.4]{Kuhr6}). \\
Define the maps $v_i:A\longrightarrow A$, $i=1,2,3,4,5$ as in the table below:
\[
\begin{array}{c|cccccc}
 x & 0 & a & b & c & d & 1 \\ \hline
v_1(x) & 0 & 0 & 0 & 0 & 0 & 1 \\
v_2(x) & 0 & 0 & 0 & c & d & 1 \\
v_3(x) & 0 & 0 & 0 & d & d & 1 \\ 
v_4(x) & 0 & a & a & c & d & 1 \\ 
v_5(x) & 0 & a & b & c & d & 1 
\end{array} 
.
\]

One can easily check that $VTO(A)=\{v_1,v_2,v_3,v_4,v_5\}$. \\
Consider $Q=\{0,c,d,1\}\subseteq A$. We can see that $(Q,\le)$ is linearly ordered, so it is a lattice.    
Then $(Q,\wedge,\vee,\odot,\ra,\rs,0,1)$ is a pseudo-MTL algebra, where 
$\odot$ is given in the table below. 

\[
\hspace{10mm}
\begin{array}{c|ccccc}
\odot & 0 & c & d & 1 \\ \hline
0 & 0 & 0 & 0 & 0 \\ 
d & 0 & d & d & c \\ 
c & 0 & d & d & d \\ 
1 & 0 & c & d & 1
\end{array}
\hspace{10mm} 
\begin{array}{c|ccccc}
x & 0 & c & d & 1 \\ \hline
v_1' & 0 & 0 & 0 & 1 \\ 
v_2' & 0 & c & d & 1 \\ 
v_3' & 0 & d & d & 1 
\end{array}
. 
\]

We have $SVTO_Q(A)=\{v'_1, v'_2, v'_3\}$, with $v'_1, v'_2, v'_3$ from the table above. \\         
Obviously $v'_1=v{_{1_{\mid Q}}}$, $v'_2=v{_{2_{\mid Q}}}=v{_{4_{\mid Q}}}=v{_{5_{\mid Q}}}$, $v'_3=v{_{3_{\mid Q}}}$. 
\end{example}

\bigskip

\section{Conclusions}

In this paper we generalize to the case of pseudo-BCK algebras the notions and results of very true operators 
which have been proved for other fuzzy logic algebras such as MV-algebras, effect algebras, R$\ell$-monoids, MTL-algebras, residuated lattices and equality algebras. 
We define and investigate the very true deductive systems and the very true pseudo-BCK homomorphisms, and we prove 
some special results for the case of some classes of pseudo-BCK algebras. 
For a very true bounded pseudo-BCK algebra $(A,v)$, we define the pseudo-BCK$_{vt,st}$ algebra by adding two truth-depressing hedges operators associated with $v$. \\
As another direction of research, one could extend these results to the case of more general algebras of 
fuzzy logic, such as pseudo-BE algebras (\cite{Bor2}) and pseudo-CI algebras (\cite{Rez2}).

\bigskip

\section* {\bf\leftline {Compliance with Ethical Standards}}

\noindent Confict of interest: The author declares that she has no conflict of interest. \\
Ethical approval: This article does not contain any studies with human participants or animals performed 
by the author.

\bigskip

\noindent {\footnotesize
\begin{minipage}[b]{10cm}
Lavinia Corina Ciungu\\
Department of Mathematics \\
University of Iowa \\
14 MacLean Hall, Iowa City, Iowa 52242-1419, USA \\
Email: lavinia-ciungu@uiowa.edu
\end{minipage}}

\end{document}